\documentclass[12pt,leqno,oneside]{amsart}

\usepackage{amssymb,amsmath,amsthm, nccmath}
\usepackage[latin1]{inputenc}
\usepackage{graphicx,color}
\usepackage[dvipsnames]{xcolor}
\usepackage{stackrel}

\usepackage[toc,page]{appendix}
\usepackage{comment}

\usepackage[left=2cm,top=2.5cm,bottom=2.5cm,right=2cm]{geometry}
\usepackage[english]{babel}

\usepackage{tikz-cd} 
\usepackage[percent]{overpic}

\def \R {{\mathbb R}}

\def \bx {\bar\xi}
\def \eqdef {\doteq}

\newcommand {\beqn} {\begin{equation*}}
	\newcommand {\eeqn}	{\end{equation*}}
\newcommand {\beq} {\begin{equation}}
	\newcommand {\eeq}	{\end{equation}}

\newtheorem{theorem}{Theorem}
\newtheorem{lemma}[theorem]{Lemma}

\newtheorem{definition}[theorem]{Definition}

\newtheorem{remark}[theorem]{Remark}

\title{Exploration of billiards with Keplerian potential}
\author{Vivina L. Barutello, Anna Maria Cherubini, Irene De Blasi}

 \thanks{\noindent Work partially supported by INdAM groups G.N.A.M.P.A. and G.N.F.M.\\
Supported by Italian Research Center on High Performance Computing Big Data and Quantum
Computing (ICSC), project funded by European Union - NextGenerationEU - and National
Recovery and Resilience Plan (NRRP) - Mission 4 Component 2 within the activities of Spoke 3
(Astrophysics and Cosmos Observations)}



\begin{document}
	
\maketitle
\begin{abstract}
We study  a class of elliptic billiards with a Keplerian potential inside, considering two cases: a reflective one, where the particle reflects elastically on the boundary, and a 
refractive one, where the particle can cross the billiard's boundary entering a region 
with a harmonic potential.  In the latter case the dynamics is therefore given by  concatenations of inner and outer arcs, connected by a refraction law. In recent papers (e.g.  \cite{deblasiterraciniellissi, IreneSusNew, IreneSusViNEW, takeuchi2021conformal}) these billiards have been extensively studied in order to identify which conditions give rise to either regular or chaotic dynamics. In this paper we complete the study by analysing  the non focused reflective case, thus complementing the results obtained in \cite{takeuchi2021conformal}  in the focused one. We then analyse the focused and non focused refractive case, where no results on integrability are known except for the centred circular case, by providing an extensive numerical analysis. 
We present also a theoretical result regarding the linear stability of homothetic equilibrium orbits in the reflective case for general ellipses, highlighting the possible presence of bifurcations even in the integrable framework.
\end{abstract}

\section{Introduction}\label{sec:intro}
Billiards with potentials can be considered a generalisation of classical Birkhoff billiards, and in the last decades they have attracted the attention of a wider and wider community of mathematicians (see \cite{jacobi1866,boltzmann1868,wojciechowski1985integrable,korsch1991new, kozlov1991billiards,Tabbook,Bol2017,pustovoitov2021topological}). As a matter of fact, many results that hold for the classical case still persist when inside the billiard a non-constant potential is considered; on the other hand, a new phenomenology, sometimes unexpected, can appear. 
This is the case discussed in \cite{IreneSusViNEW}, where a particular class of billiards with a Keplerian potential is proved to be chaotic when the domain is a  {\it centred} ellipse (actually the results holds in a much more general class of planar domains);
on the contrary, when we deal with classical Birkhoff billiards, it is well known that  {\it any} elliptic boundary leads to an integrable dynamics (see \cite{birkhoff1927acta}). 

Within all possible non-constant potentials, we consider the purely Keplerian one, generated by a mass located at the origin, which is in the interior of a bounded region of the plane. 
To be more precise, let $D\subset\R^2$ be a bounded set such that $(0,0) \in \mathring{D}$ and consider an inner potential given by
\begin{equation*}\label{eq:potinterno}
	V_I(z) = h_I +\frac{\mu}{|z|}, \quad  z \in D,
\end{equation*}
where $\mu>0$ is the mass parameter of a Keplerian centre located at the origin. The quantity $h_I$, which represents the total inner energy, can take in principle any real value; nevertheless, since the results in \cite{IreneSusViNEW} hold for high values of the inner energy, we restrict to the case $h_I>0$.
Hence, the motion inside $D$ consists in branches of Keplerian hyperbol\ae\, which, in the models considered, interact with the boundary $\partial D$ in two different ways: the first possibility is an elastic reflection (see Figure \ref{fig:rifl-rifr}, left). In this case the tangent component of the velocity is kept constant, while the normal one changes in sign. This first case corresponds to the well known \emph{reflective Keplerian billiard} (or Kepler billiard, see for instance \cite{takeuchi2021conformal}, and for a more general treatment of this kind of models \cite{bolotin2000periodic,Bol2017,shilnikov1967}). 
\begin{figure}
 \centering
\includegraphics[width=0.4\linewidth]{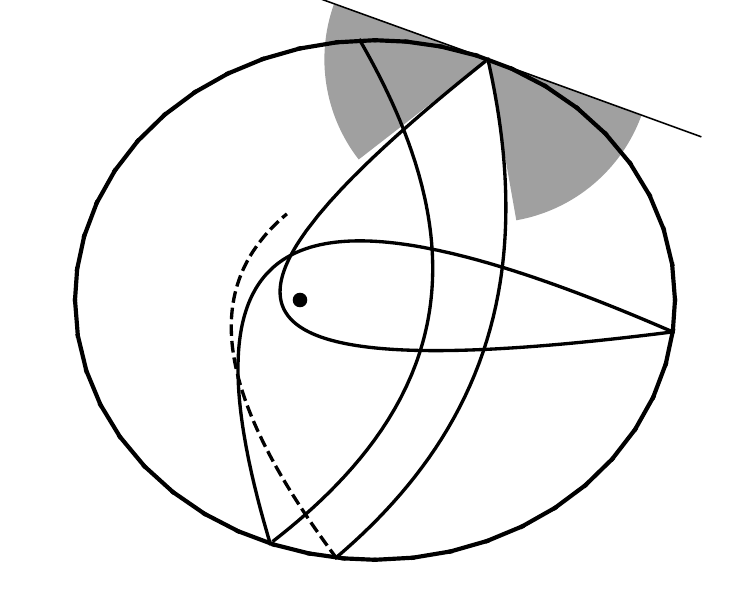} \qquad 
\begin{overpic}[width=0.4\linewidth]{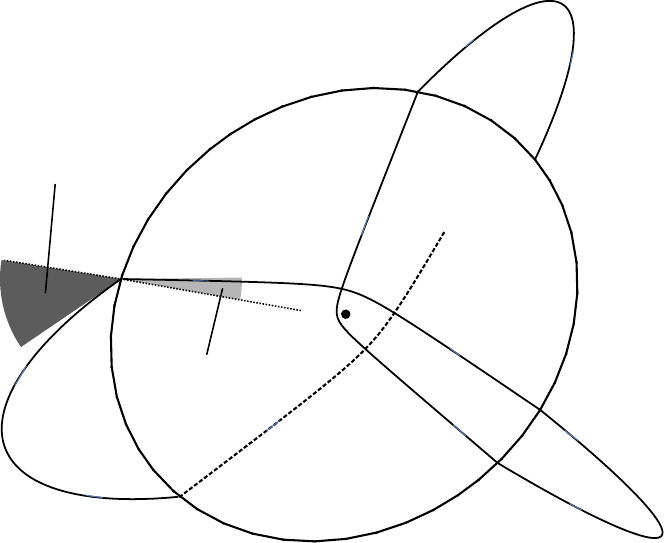}
	\put (7,55) {$\alpha_E$}
	\put (28,25) {$\alpha_I$}
\end{overpic}
	\caption{In the left picture a Keplerian reflective billiard; on the right, a refractive billiard as described in \cite{IreneSusViNEW}, where the angles $\alpha_E$ and $\alpha_I$ refer to Eq. \ref{eq:snell}.}
	\label{fig:rifl-rifr}
\end{figure}

The second model we consider is more complex. The particle can exit the domain entering in the exterior region, where it is subjected to a Hooke force induced by the potential 
\begin{equation}
	V_E(z)=h_E-\frac{\omega^2}{2}|z|^2, \quad z\notin D, 
\end{equation}
where $h_E, \omega>0$. The trajectories described outside $D$ are then harmonic ellipses and the transition inside-outside or outside-inside is governed by the refraction Snell's law 
\begin{equation}\label{eq:snell}
	\sqrt{V_I(z)}\sin\alpha_I=\sqrt{V_E(z)}	\sin\alpha_E, 
\end{equation} 
where $z\in\partial D$ is the transition point, and $\alpha_E, \alpha_I$ are the angles of respectively the outer and inner arc with the normal vector to $\partial D$ in $z$ (see Figure \ref{fig:rifl-rifr}, right). We will refer to this case as the {\it refractive Keplerian} one. Such model  is justified by physical reasons: the system so constructed can be used to describe the motion of a particle inside an elliptic galaxy with a central mass (see for instance \cite{Delis20152448,deblasiterraciniellissi,IreneSusNew}). While in the reflective case, we took $h_I>0$, here we will focus on the case $h_E<h_I$: in \cite{IreneSusNew}, we indeed noticed an interesting variety of phenomena for $h_I\gg h_E$. Of course, a more general analysis with real values of $h_E, h_I$ is possible.   \\

The present work has to be intended as a complement of the analysis carried on in \cite{IreneSusViNEW}, inspired by the interaction of the authors with colleagues in the field, who raised very interesting open questions about the models. For this reason, let us now summarise, for general domains, hypotheses and main result of \cite{IreneSusViNEW}.
Assume that the bounded domain $D$ containing the origin has a simple $C^2$ boundary parameterised by $\gamma = \gamma(\xi)$, $\xi \in [a,b]$, with $\gamma(a) = \gamma(b)$.   

\begin{definition}\label{def:cc_intro}
	A \emph{central configuration} (in short, c.c.) is a value of the parameter $\bar \xi \in [a,b]$ such that 
	\begin{itemize}
		\item $\overline P=\gamma(\bar \xi)$ is a constrained critical point for $|\cdot|_{|_{\partial D}}$, that is, the position vector $\overline P$ is orthogonal to the boundary $\partial D$ at $\overline P$;
		\item the half-line connecting the origin to $\overline P$ intersects $\partial D$ only at $\overline P$.
	\end{itemize}
	A central configuration $\bar\xi$ is termed \emph{non degenerate} if the second differential of the function $|\cdot|_{|_{\partial D}}$ is not degenerate at $\overline P$.\\
	The domain $D$ is termed \emph{admissible} if there exist at least two non degenerate\footnote{Let us highlight that, for simplicity, the definition of admissible domain is slightly stronger than in \cite{IreneSusViNEW}. } central configurations corresponding to a pair of points on $\partial D$ which are not aligned with the origin (i.e. that are \emph{not antipodal}).
\end{definition}
Central configurations are of fundamental importance for the dynamics in both reflective and refractive case. Along the direction of $\bar P$, one has indeed equilibrium trajectories for the systems, called \emph{homothetic motions}: they reflect on the mass on one side (through a regularisation process, see \cite{Levi-Civita}), and go back and forth along the half-line identified by $\overline{0P}$.   \\
{In Figure \ref{fig:zoo} we propose some examples of different domain's shapes, starting from a circle considering increasingly complex figures. In every quadrant, it is highlighted either with dotted lines or coloured areas the regions where, putting the mass $\mu$, one obtains non-admissible domains\footnote{{We stress that it is completely equivalent to fix the mass at the origin and moving the domain (keeping the origin in its interior) or to fix the domain and moving the inner mass. In the first case the potential remains the same and the domain changes; in the second one we have to change the potential. For computations, the first point of view is simpler, while to visualise the problem (as in Figure \ref{fig:zoo}) the second one results more convenient.}}: it is easily verified that, moving the origin in those points, one has only two antipodal central configurations or, alternatively, only degenerate ones. We highlight that not all the domains in Figure \ref{fig:zoo} have a $C^2$ boundary: in any case, the notion of admissibility can be extended to billiards with piecewise $C^2$ boundaries.}    \\

\begin{figure}
\centering
\begin{overpic}[height=0.15\textheight]{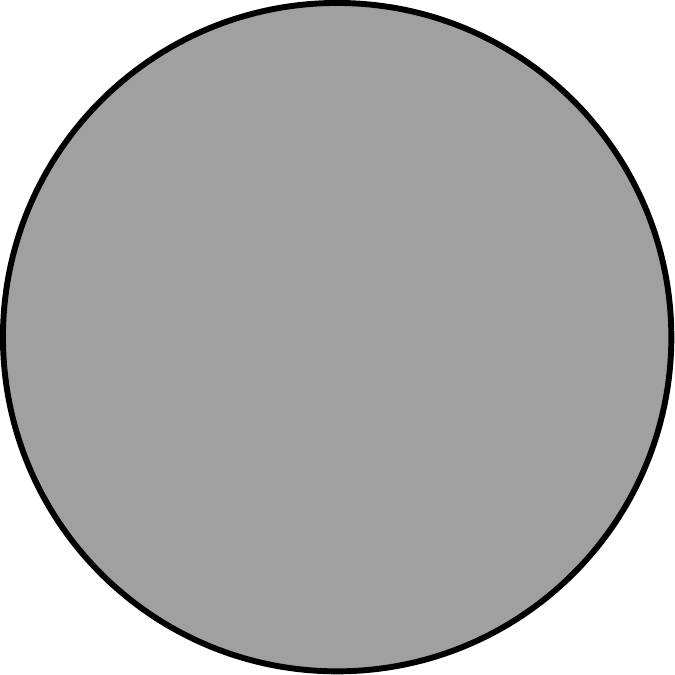}
	\put(0,95){(A)}
\end{overpic}\quad
\begin{overpic}[height=0.15\textheight]{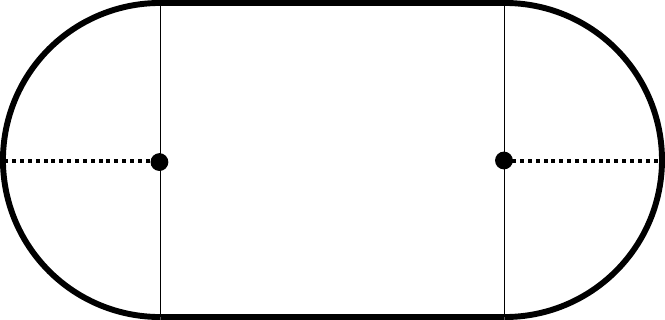}
	\put(0,46){(B)}
\end{overpic}\quad
\begin{overpic}[height=0.15\textheight]{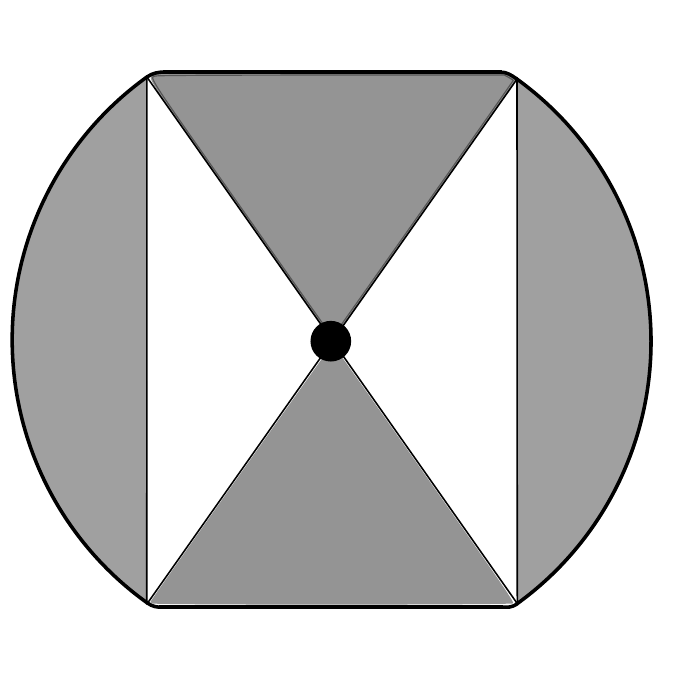}
	\put(0,95){(C)}
\end{overpic}\\
\quad\\
\begin{overpic}[height=0.15\textheight]{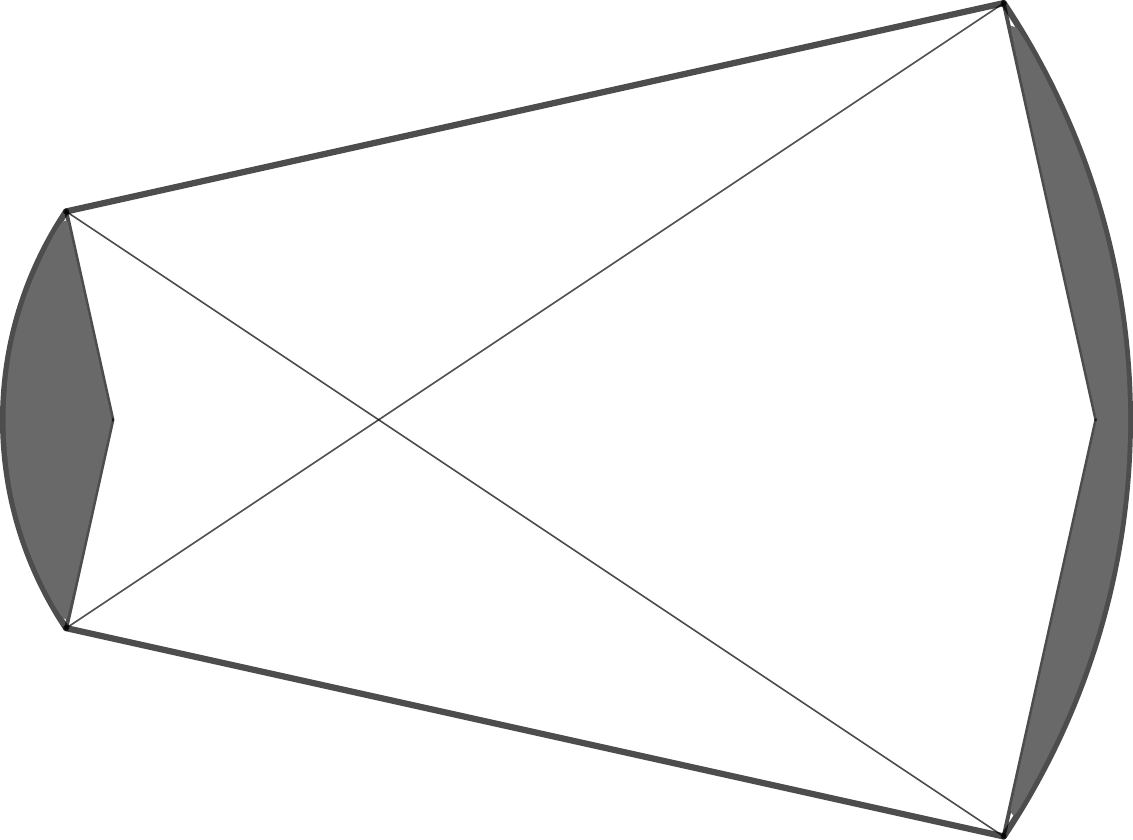}
	\put(0,70){(D)}
\end{overpic}\qquad
\begin{overpic}[height=0.15\textheight]{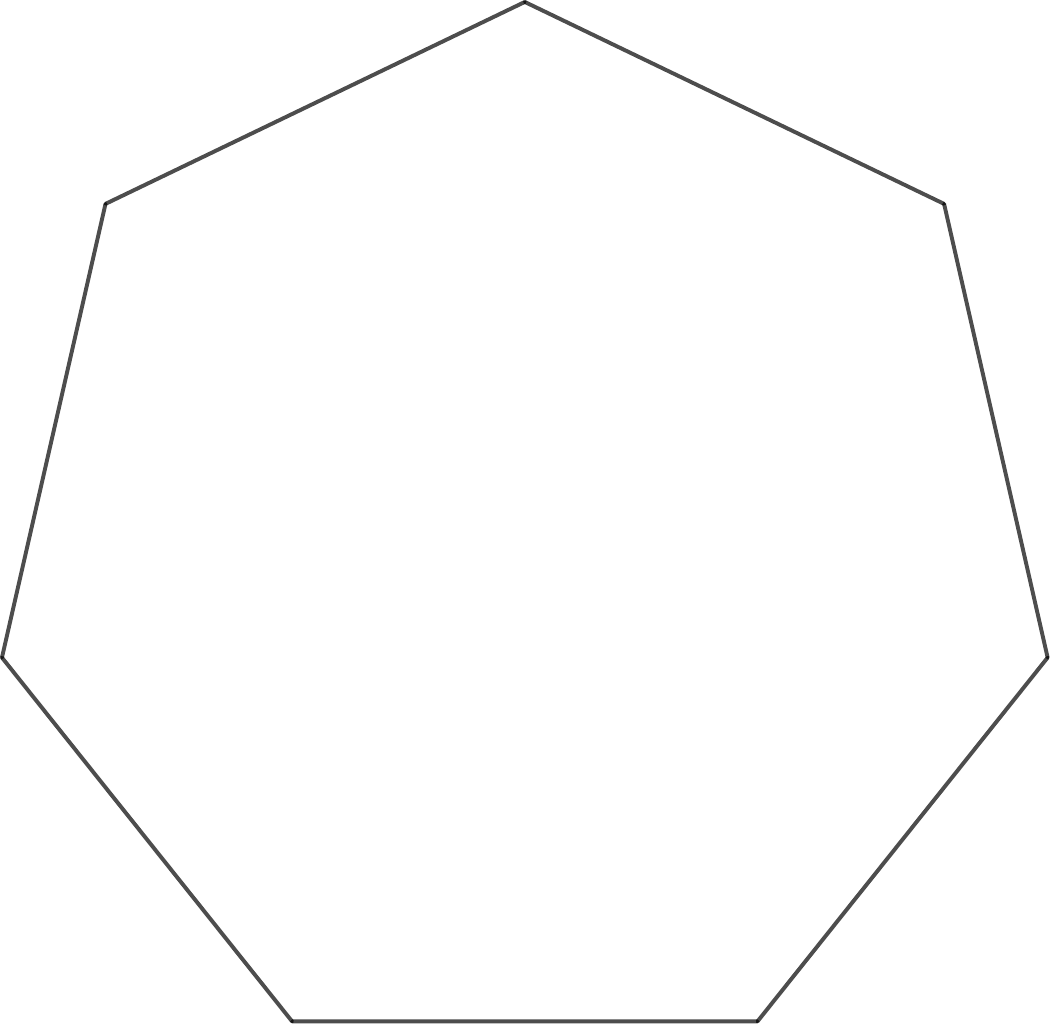} 
\put(0,93){(E)}
\end{overpic}\qquad
\begin{overpic}[height=0.15\textheight]{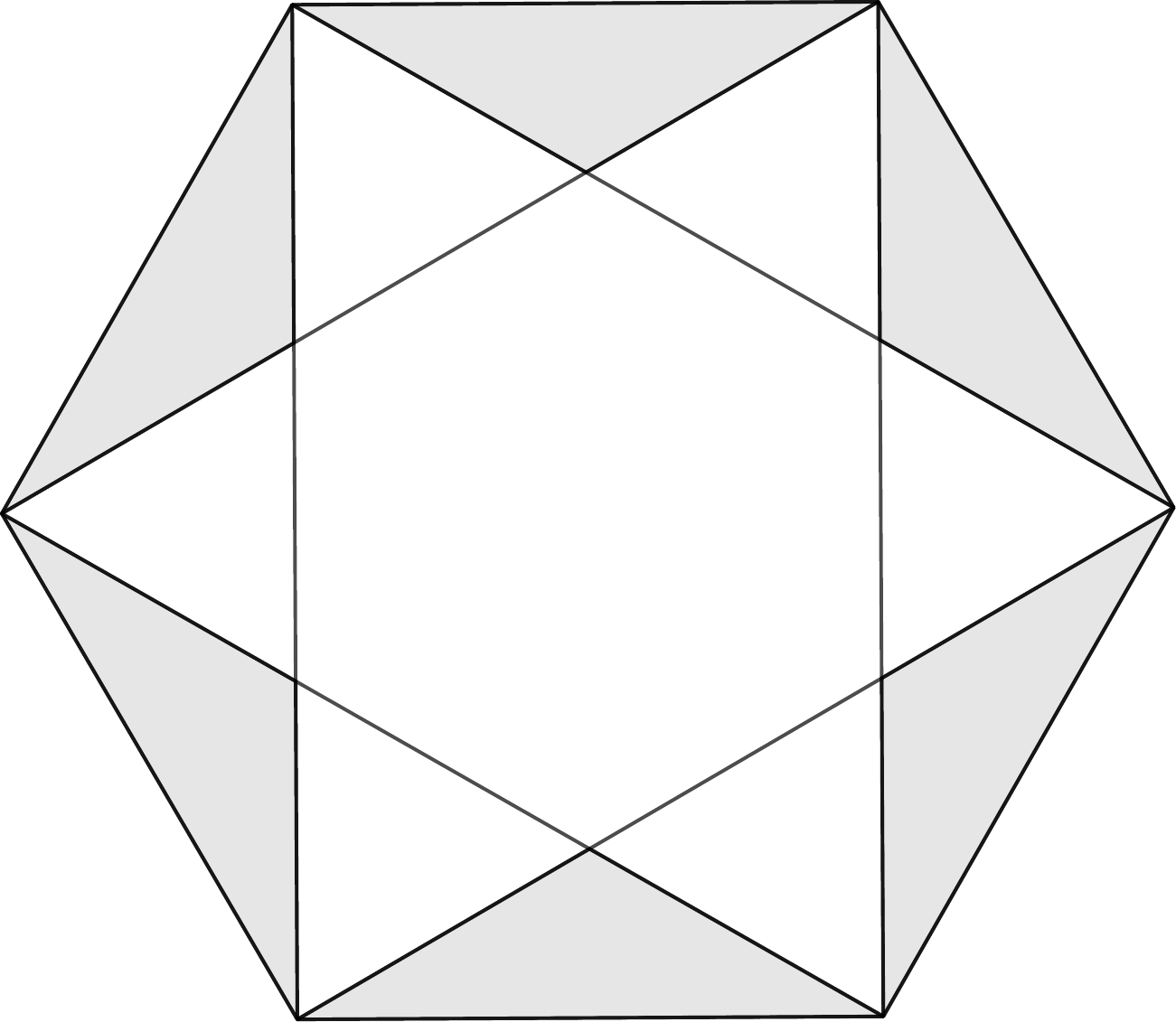}
	\put(0,83){(F)}
\end{overpic}
\caption{Examples of domain and possible positions of the Keplerian mass leading to non-admissible domains (dotted lines or coloured regions). (A) Circle: if the circle is centred at the origin, every point of the boundary is a degenerate c.c.; if the centre is in any other point there are two non-degenerate but antipodal c.c. The circle is then never admissible. (B) Stadium-shape domain, where two half-circles are connected by a rectangle. Here, the position of the origin for which we have non-admissibility reduces to the two dotted lines (this case is very similar to the elliptic one, see Theorem \ref{thm:primo_teorema}).(C) Strip of circle: the two circular arcs have the same radius and centre. If the centre of the circle is at the origin we have two antipodal c.c. and indefinitely many degenerate ones; when $\mu$ is located in the grey region, there are just two antipodal c.c. (D) Trapezoidal shape: the circular edges have same centre but different radii. The segments bounding the grey regions are orthogonal to the straight edges. (E) Regular polygon with an odd number of vertices: these domains are always admissible. (F) Regular polygon with an even number of vertices: there is always a non-zero measure region where the mass can not be placed to have admissibility.}
\label{fig:zoo}
\end{figure}

\begin{theorem}[\cite{IreneSusViNEW}, Theorem 1.4]\label{thm:main_lavoro}
	Let $D$ be an admissible domain. Then, if $h_I$ is large enough, both the reflective and refractive case admit a topologically chaotic subsystem. 
\end{theorem}

The chaoticity in the previous theorem is intended as a conjugation with the Bernoulli shift map obtained by constructing a suitable symbolic dynamics (see \cite{Dev_book}), \textit{shadowing} concatenations of homothetic motions. \\
We observe that a centred ellipse (namely, with the centre at the origin) is an admissible domain since it has four central configurations corresponding to the intersection of the two axes with the boundary. Hence, any centred ellipse is chaotic in both types of Keplerian billiards, at least for large enough inner energies. \\
On the other hand, in (\cite{takeuchi2021conformal}) Takeuchi and Zhao show that a focused ellipse (i.e. an elliptic Kepler billiard with the mass at one of the two foci) is integrable in the reflective case. It is then clear that, keeping the mass at the origin and simply translating the ellipse, there is a transition from integrability to chaos and vice versa. This notable bifurcation phenomenon, obtained simply by a translation, motivates a further analysis on the actual holding of the admissibility condition in relation with the position of the Keplerian centre inside the billiard. {The great interest on elliptic domains in the classical setting encouraged us to look at this direction also in the non-trivial potential framework (see for instance \cite{fedorov2001ellipsoidal,pustovoitov2019topological,kobtsev2020elliptic,takeuchi2021conformal}).}
In the first analytical result (see Theorem \ref{thm:primo_teorema}) we characterise completely the admissibility of the elliptic case in terms of the position of the central mass. In particular, the ellipse is admissible for every position of the latter, except for a zero-measure set consisting of two or four, depending on the eccentricity, disjoint segments. 
Let us outline that admissibility is just a sufficient condition to have chaos and not a necessary one: in Section \ref{subsec:num_chaos} we propose some numerical investigations that show the presence of chaos also in the absence of admissibility, at least for high inner energies. We stress that such section is particularly interesting since it presents the numerical analysis of cases not covered by Theorem \ref{thm:main_lavoro}; in particular, it presents evidences of chaotic behaviour \emph{even in the focused elliptic case} for the refractive dynamics. 

The second analytical result we prove (Theorem \ref{thm:DeltaKep}) concerns the linear stability of homothetic trajectories in the reflective case (the analogous for the refractive billiard is stated in \cite{deblasiterraciniellissi}). In brief, Theorem \ref{thm:DeltaKep} claims the existence of numerical constants, depending on the physical parameters of the problem, whose sign determines the stability of the homothetics. For simplicity, we consider the central configurations of an ellipse with the mass on the horizontal axis, covering both the integrable and the non-admissible case, although the technique we used can be easily generalised to any domain.  In the case considered, the domain admits two horizontal homothetics:  one  is always a saddle, while the other becomes a second saddle for $h_I$ sufficiently large. It is well-known that the presence of multiple saddles and heteroclinic connections between them can be considered as chaos precursors: in this sense this second result is connected to the first one. On the other hand, also in the integrable case such result contains a novelty with respect of the current literature, giving an explicit and easy-to-verify condition for the stability of the equilibrium trajectories. 
Theorem \ref{thm:DeltaKep} highlights a fundamental difference between Kepler and classical Birkhoff billiards: \emph{while in the first case the stability of the homothetics trajectories can change with the physical parameters, this does not happen in the classical case, where such stability is the same for any energy level.} \\
In Section \ref{subsec:bif} we use the above results to test our numerical routines (produced with \emph{Mathematica}$\,^\copyright$), giving practical examples of bifurcation phenomena and obtaining a very good coherence of numerical and analytical expected results. The examples we propose are focused or close-to-focused ellipses and close-to-centred circles, being the focused ellipse and the centred circle two integrable cases: we highlight that centred circles are \emph{highly degenerate}, since both the reflective and refractive dynamics reduce in this case to a simple translation map (see \cite{IreneSusNew}). 

\section{Analytical results}
In this section, we will present two analytical results regarding elliptic domains: the first concerns a finer analysis of the condition of admissibility in Definition \ref{def:cc_intro} already introduced in \cite{IreneSusViNEW}, and holds for both the reflective and refractive case. The second one deals with the linear stability of homothetic trajectories in the reflective case: an analogous result for the refractive case is presented in \cite{deblasiterraciniellissi}. \\
Since the potentials involved in our models are radial, hence rotationally invariant, it is not restrictive to consider ellipses with axes in the $x$ and $y$ direction: then, given $(x_0,y_0) \in \mathbb{R}^2$, we will consider curves parameterised by 
\begin{equation}\label{eq:ellipse}
	\gamma(\xi)=(\cos\xi+x_0, b\sin\xi+y_0), 
	\quad \xi \in [0,2\pi],
\end{equation}
where, without loss of generalisation, the semi-major axis is fixed to 1 and $b\in[0,1)$ is the length of the semi-minor one; we stress that $b \neq 1$: the circular case will be considered at the end of next Section. \\
 The point $(x_0, y_0)\in(-1,1)\times(-b, b)$ is the centre of the ellipse and, for the origin to be inside the domain,  it must satisfy
\begin{equation}\label{eq:cond_centre}
	x_0^2 + \frac{y_0^2}{b^2} <1.
\end{equation}

\subsection{Admissibility of elliptic domains}\label{ssec:adm}
In the case of an elliptic domain, it is possible to find explicit conditions linking the position of the centre $(x_0, y_0)$ to the notion of admissibility, according to Definition \ref{def:cc_intro}. 
\begin{lemma}\label{lem:conto}
	An elliptic domain $D$ whose boundary is parameterised by \eqref{eq:ellipse} and \eqref{eq:cond_centre} is not admissible if and only if
	\begin{equation}\label{eq:condChaos}
		\begin{cases}
			x_0=0\\
			\displaystyle |y_0|\geq \frac{1-b^2}{b}
		\end{cases}
	\text{ or }
	\begin{cases}
		y_0=0\\
		|x_0|\geq 1-b^2
	\end{cases}
	\end{equation}
\end{lemma}
\begin{figure}
	\centering
	\begin{overpic}[height=0.18\textheight]{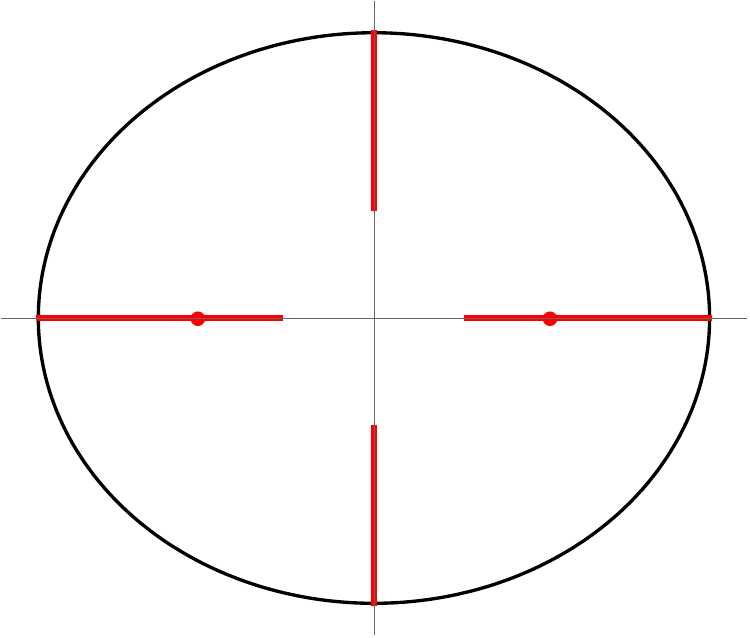}
		\put (50,45) {\rotatebox{0}{\tiny{$(x_0,y_0)$}}}
	\end{overpic}\qquad
	\begin{overpic}[height=0.18\textheight]{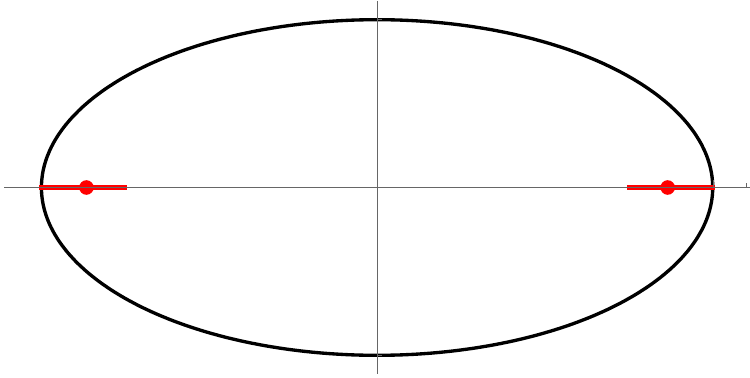}
		\put (51, 22) {\rotatebox{0}{\tiny{$(x_0,y_0)$}}}
	\end{overpic}
	\caption{Ellipses centred in a point $(x_0, y_0)$ satisfying condition \eqref{eq:condChaos}. Left: $b>1/\sqrt{2}$; right: $b<1/\sqrt{2}$. The red lines mark the positions of the origin for which the corresponding domain is not admissible.The red dots are in correspondence of the ellipses' foci. }
	\label{fig:segments}
\end{figure}

\begin{remark}
	Studying conditions \eqref{eq:condChaos} we deduce that $D$ is not admissible if and only if the centre of the ellipse belongs to at most two pairs of segments on the $x$ and $y$ axes, whose existence, location and length depends on $b$ (see Figure \ref{fig:segments}). The pair lying on the $x$ axis exists for any value of $b$; while the vertical one disappears whenever $b\sqrt{2}/2$. Indeed, in this case the threshold value $(1-b^2)/b$ is bigger than $b$.
	
	\noindent Moreover, we observe that the origin, and hence the Keplerian mass, is situated at one of the two foci of the ellipse if and only if $(x_0,y_0) = (\pm\sqrt{1-b^2}, 0)$. According to Lemma \ref{lem:conto}, in this case the domain is not admissible and then our result is consistent with \cite{takeuchi2021conformal} where this case is proved to be integrable for the reflective dynamics.
\end{remark}

\begin{proof}[Proof of Lemma \ref{lem:conto}]
	Our aim is to prove that {the function $\xi \to |\gamma(\xi)|$, $\xi \in [0,2\pi]$, admits at least two non degenerate critical points corresponding to non-antipodal direction if and only if condition \eqref{eq:condChaos} is not satisfied. Since $b\neq 1$ and  $|\gamma(\cdot)|$ is smooth, periodic and non-constant Weierstrass Theorem guarantees that there are at least two non degenerate central configurations.} 
	
	Let us start assuming that $x_0 y_0 \neq 0$: this means that the Keplerian centre, located at the origin, does not lie on any axis of the ellipse. In this case, by the symmetry properties of an ellipse, two central configurations cannot be aligned with the origin and the admissibility of $D$ follows.
%
	
	Let now consider the case $x_0y_0 =0$ Critical points of $|\gamma(\cdot)|$ are zeroes of the function
	\begin{equation*}
			f(\xi) \eqdef -\frac12 \frac{d}{d\xi}|\gamma(\xi)|^2 =  (1-b^2)\cos\xi\sin\xi+x_0\sin\xi-by_0\cos\xi. 
	\end{equation*}
When $y_0=0$, we have then
\begin{equation*}
	f(\xi)=\sin\xi \left[(1-b^2)\cos\xi+x_0\right],
\end{equation*}
which admits two antipodal zeroes in $\xi=0$, and  $\xi=\pi$. Moreover, whenever $|x_0| < 1-b^2$, it has two additional zeroes when $\xi=\pm \arccos\left(-x_0/(1-b^2)\right)$. In such case, the domain $D$ is admissible, while when $|x_0| \geq 1-b^2$ it is not. \\
Finally, when $x_0=0$, one has 
\begin{equation*}
	f(\xi) = \cos\xi\left[(1-b^2)\sin\xi-by_0\right], 
\end{equation*}
and then again there are two antipodal zeroes in $\pm\pi/2$ and two more in $\xi=\arcsin(by_0/(1-b^2))$ and $\xi=\pi-\arcsin(by_0/(1-b^2))$ only when $|y_0|<(1-b^2)/b$. 
\end{proof}

From Lemma \ref{lem:conto} and from the main result in \cite{IreneSusViNEW} immediately follows that \emph{both reflective and refractive Keplerian elliptic billiards are chaotic for large enough inner energies for almost every choice of the mass' position}. More precisely, we have

\begin{theorem}\label{thm:primo_teorema}
Let $D$ be an elliptic domain with different semi-axes and consider the dynamics of a reflective/refractive Keplerian billiard in $D$. Then there exists a subset $\widetilde D \subset D$ such that:
\begin{itemize}
	\item for any $P \in \widetilde D$, if the Keplerian mass is located  at $P$, then the dynamics is chaotic for large enough inner energies;
	\item $D \setminus \widetilde D$ has zero-measure and 
	the foci of $\partial D$ belong to $D \setminus \widetilde D$.
\end{itemize}
\end{theorem}

To conclude, let us make two comments.
First of all, we observe that as far as the elliptic domain degenerates in a circle then the set of centre's positions for which the domain is not admissible suddenly becomes {\it the whole circle}; indeed, when the circle is centred at the Keplerian mass, there are infinitely many central configurations, but all of them are degenerate (in this case the billiard is integrable both in the reflective and refractive case, see \cite{IreneSusNew}). On the other hand, whenever the centre of the circle is displaced from the origin, there are exactly two central configurations which are non degenerate but antipodal, making the domain not admissible. 
Let us remark again that admissibility is just a sufficient condition to have chaos and not a necessary one: it is in principle possible to have chaos even in the absence of admissibility, and in Section \ref{subsec:num_chaos} we propose some numerical examples. 

\subsection{Linear stability of homothetic trajectories in the reflective case}

Let us consider an ellipse parameterised by \eqref{eq:ellipse} with $y_0 = 0$. In this case $\bar \xi =0$ and $\bar \xi =\pi$ are central configurations and hence homothetic equilibrium trajectories for the system, both in the reflective and in the refractive case. In this paragraph we analyse their linear stability in the reflective case (in the refractive one it has been already done in \cite{deblasiterraciniellissi}). Let us highlight that we focus on this special situation (centre of the ellipse on the horizontal axis) mainly for two reason: first, computations are much simpler that in the general case, although the general case can be treated by following the same scheme. On the other hand, it is particularly interesting from a dynamical point of view: indeed, moving the centre on the horizontal axis, we can cover both a chaotic and an integrable regime.
Moreover, this analysis allows to detect the presence of multiple saddles and the possible heteroclinic connections between them is a well known chaos indicator (see for instance \cite{koz1983}).\\

Before stating the main result of this section, we need to define the first return map associated to our dynamics: here, we will briefly recall the main definition, while the interested reader can find more details in \cite{IreneSusViNEW}.\\
Let us start with the reflective case, and consider initial conditions $(p_0,v_0)\in \partial D\times \R^2$ for an inner Keplerian arc at energy $h_I$, namely, such that $v_0$ points inside $D$ and $|v_0|^2/2+V_I(p_0)=0$ (see also Figure \ref{fig:primo_ritorno}, left). 
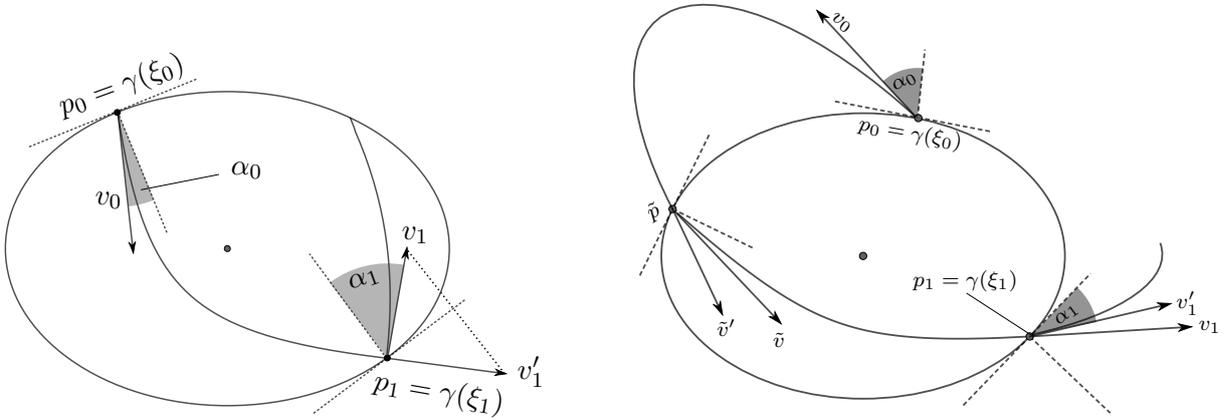
\begin{figure}[h!]
\begin{overpic}[height=0.18\textheight]{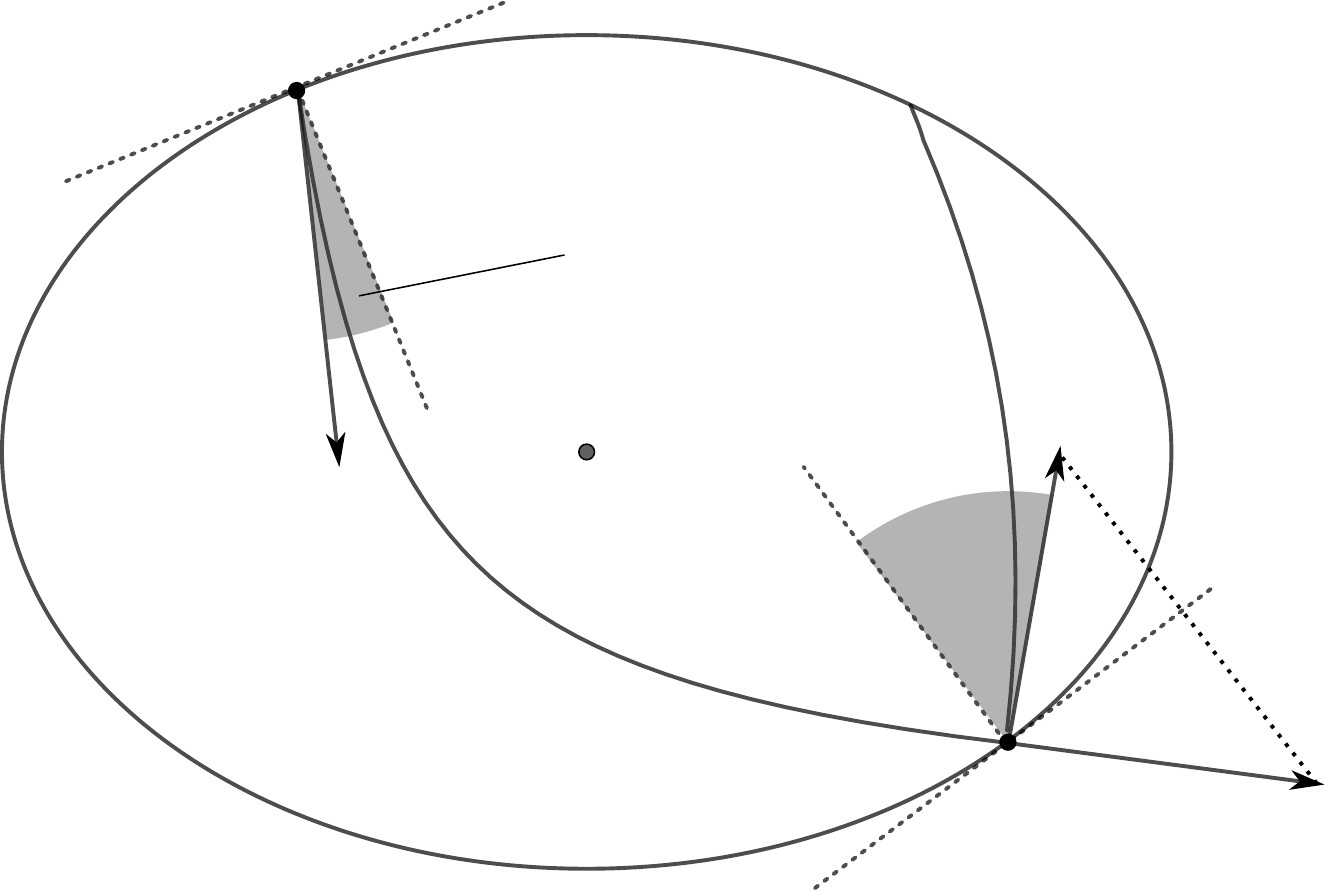}
	\put (10,60) {\rotatebox{20}{\small{$p_0=\gamma(\xi_0)$}}}
	\put (18,42) {\small{$v_0$}}
	\put (45,48) {\small{$\alpha_0$}}
	\put (73,5) {\rotatebox{-10}{\small{$p_1=\gamma(\xi_1)$}}}
	\put (68,25) {\rotatebox{20}{\small{$\alpha_1$}}}
	\put (102,7) {\rotatebox{0}{\small{$v'_1$}}}
	\put (79,35) {\rotatebox{0}{\small{$v_1$}}}
\end{overpic}\qquad\qquad
\begin{overpic}[height=0.22\textheight]{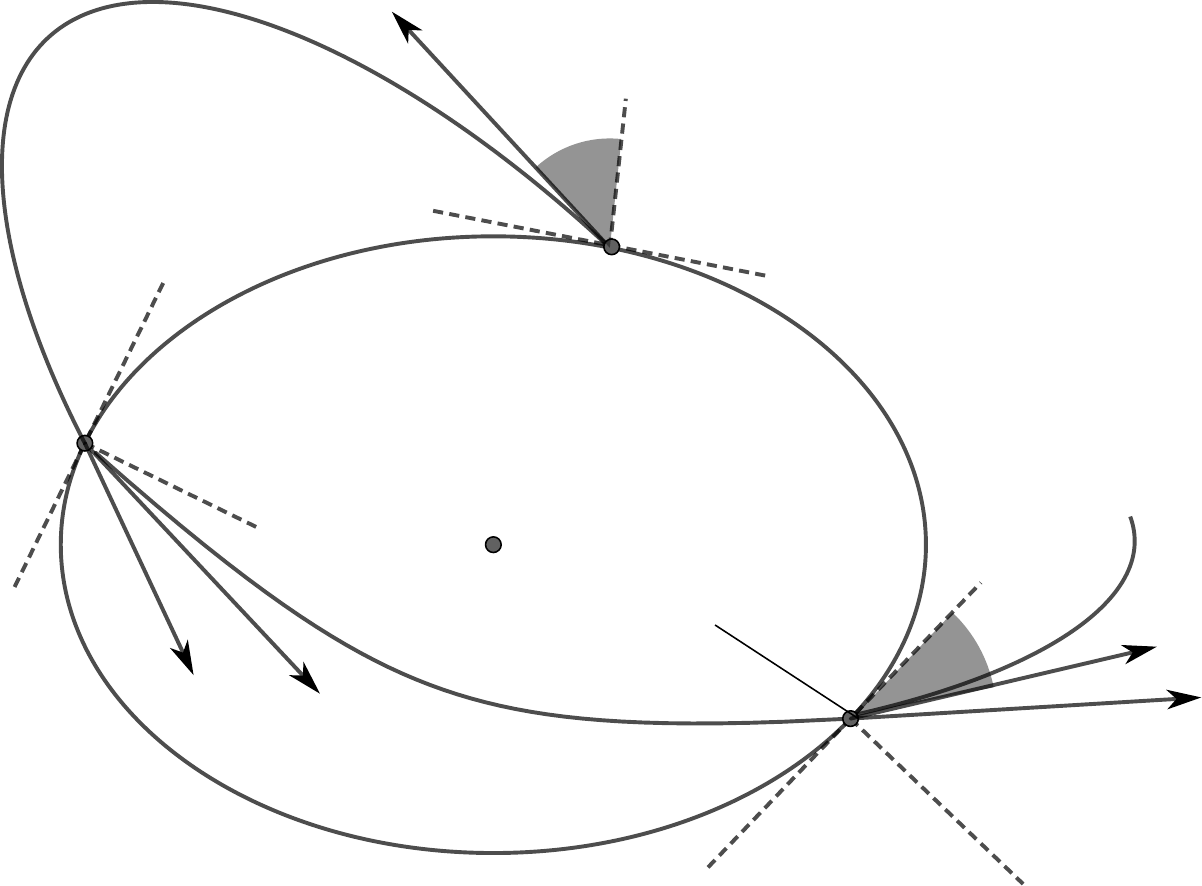}
	\put (46,58){\rotatebox{30}{\tiny{$\alpha_0$}}}
	\put (40,51){\rotatebox{-10}{\tiny{$p_0=\gamma(\xi_0)$}}}
	\put (35,71){\rotatebox{-35}{\tiny{$v_0$}}}
	\put (2,35){\rotatebox{30}{\tiny{$\tilde p$}}}
	\put (15,14){\rotatebox{0}{\tiny{$\tilde v'$}}}
	\put (25,12){\rotatebox{0}{\tiny{$\tilde v$}}}
	\put (50,23){\rotatebox{0}{\tiny{$p_1=\gamma(\xi_1)$}}}
	\put (75,16){\rotatebox{37}{\tiny{$\alpha_1$}}}
	\put (97,19){\rotatebox{0}{\tiny{$v'_1$}}}
	\put (101,15){\rotatebox{0}{\tiny{$v_1$}}}
\end{overpic}
	\caption{{Left: first return map for the reflective case. The initial conditions $(p_0,\alpha_0)$, parameterised by $(\xi_0,\alpha_0)$, are transformed into the conditions $(p_1,\alpha_1)$, corresponding to the pair $(\xi_1,\alpha_1)$. Right: first return map in the refractive case: from the initial conditions $(p_0,v_0)$, corresponding to $(\xi_0,\alpha_0)$, starts an elliptic outer arc which arrives on $\partial D$ in $\tilde p$ with velocity $\tilde v'$. Such velocity is refracted into $\tilde v$, and an inner Keplerian arcs starts from $(\tilde p,\tilde v)$. It intersects $\partial D$ again in $p_1$ with velocity $v_1'$, which is refracted into $v_1$. The final conditions $(p_1,v_1)$ correspond to the pair $(\xi_1,\alpha_1)$. }}
	\label{fig:primo_ritorno}
\end{figure}
The arc starting from $(p_0, v_0)$ will intersect again $\partial D$ in a point $p_1$ with velocity $v_1'$: if it is not tangent to the boundary, it will be reflected into $v_1$, by keeping the tangent component and inverting the normal one. At this point, $(p_1,v_1)$ can be considered as the initial conditions of a new inner Keplerian arc at the same energy, concatenated to the first one by reflection. \\
The map $(p_0,v_0)\mapsto (p_1,v_1)$ can be expressed as a two-dimensional discrete map by taking a suitable set of coordinates: recalling that $\partial D$ can be parameterised by a regular curve $\gamma:[0,2\pi)\to \R^2$, any initial condition $(p,v)$ can be univocally determined by the real numbers $(\xi,\alpha)\in [0,2\pi)\times (-\pi/2,\pi/2)$, such that $p=\gamma(\xi)$ and {$\alpha$ is the angle from the inward-pointing normal vector at $\gamma(\xi)$ to $v_0$.} We  then obtain the \textit\emph{first return map} for the reflective Kepler billiard
\begin{equation}\label{eq:intro_rit_rifl}
	F: [0,2\pi)\times \left(-\frac{\pi}{2}, \frac{\pi}{2}\right)\to [0,2\pi)\times \left(-\frac{\pi}{2}, \frac{\pi}{2}\right), \quad (\xi_0,\alpha_0)\mapsto(\xi_1,\alpha_1).  
\end{equation} 
The map $F$ is well defined whenever the arcs do not touch $\partial D$ tangentially; moreover, if $\bx$ is a central configuration according to Definition \ref{def:cc_intro}, then $(\bx,0)$ is a fixed point for $F$. Deriving analytically the linear stability of $(\bx,0)$ with respect to the map $F$ is completely equivalent to compute the linear stability of the corresponding equilibrium homothetic trajectory in the planar billiard dynamics. The map $F$ in the coordinates $(\xi,\alpha)$ will be used in Section \ref{sec:numerical} to describe the reflective dynamics for different domains' shapes and in the proof of Theorem \ref{thm:DeltaKep}. \\

Let us now pass to the refractive case: in principle, the method to construct the first return map is the same, although now we have to take into account an outer dynamics as well. In this case, we can start from $(p_0, v_0)\in \partial D\times \R^2$ initial conditions for an \emph{outer} arc, and follow the dynamics through a complete concatenation of the latter with the subsequent inner arc, as displayed in Figure 		\ref{fig:primo_ritorno}, right, by taking into account also the refractions. We obtain the pair $(p_1, v_1)$, from which one can start with a new concatenation of outer and inner arc. As in the reflective case, we can use coordinates $(\xi,	\alpha)$, where now $\alpha$ is the angle of the velocity with the outward-pointing normal vector to $\gamma$ in $\xi$, to obtain the first return map for the refractive case
\begin{equation}\label{eq:intro_rit_rifr}
		G: [0,2\pi)\times \left(-\frac{\pi}{2}, \frac{\pi}{2}\right)\to [0,2\pi)\times \left(-\frac{\pi}{2}, \frac{\pi}{2}\right), \quad (\xi_0,\alpha_0)\mapsto(\xi_1,\alpha_1).  
\end{equation}
The problem of the good definition of $G$ is slightly more complex than in the case of $F$, as it requires to take into account a particular property of Snell's law. Indeed,  whenever $h_I>h_E$ then  $V_E(p)<V_I(p)$ for every $p\in \partial D$. 
In terms of Snell's law, this translates into the fact that, while given any $\alpha_E$ it
 is always possible to find $\alpha_I$ such that Eq. \eqref{eq:snell} is verified, this is not true in the inverse case. Indeed, if we fix $\alpha_I$ the equation is solvable in the unknown $\alpha_E$ if and only if $\alpha_I$ is sufficiently small, and in particular if it is less than a \emph{critical angle}, depending on the transition point $p$: 
\begin{equation}\label{eq:crit_ang_intro}
	|\alpha_I|\leq \arcsin\left(\frac{V_E(p)}{V_I(p)}\right)=\alpha_{crit}. 
\end{equation} 
Geometrically, this means that, to be refracted outside, it is not enough that the inner arc are not tangent to the boundary: they need to be \emph{transversal enough} to it, giving a constraint to the initial conditions for which $G$ is well defined. We will see, in Section \ref{subsec:num_chaos}, that numerically such condition will define an invariant curve in the phase space $(\xi, \alpha)$, bounding the region where $G$ and its iterates are well defined. We recall that, exactly as in the reflective case, central configurations correspond to fixed points of $G$: the stability analysis of such points with respect to the domain's shape has been already carried on in \cite{deblasiterraciniellissi}. \\

{Let us now return to the reflective case, and, given $\bx$ central configuration, consider the fixed point $(\bx, 0)$. We will compute the linear stability of the latter by relying on the Implicit Function Theorem. To do that, we need to recall a special property of the solutions of the inner differential problem (see \cite[Appendix A]{IreneSusViNEW} for a more detailed explanation), seen as extremals of a suitable length functional.\\
To be more precise, let us consider two points on the boundary, parameterised by $\gamma(\xi_0), \gamma(\xi_1)$, and suppose that there exists a unique solution $z(\cdot)$ of the fixed ends problem\footnote{Although in principle the uniqueness is not guaranteed, it is possible to impose suitable topological constraints to obtain it, see \cite{IreneSusViNEW}.}
\begin{equation}
\begin{cases}
	z''(t)=\nabla V_I(z(t)) \quad & t\in [0,T]\\
	\dfrac{|z'(t)|^2}{2}+V_I(z(t))=0& t\in [0,T]\\
	z(t)\in D& t\in [0,T]\\
	z(0)=\gamma(\xi_0),\ z(T)=\gamma(\xi_1)
\end{cases}
\end{equation}
for some $T>0$. Let us define the \emph{Jacobi distance} between $\gamma(\xi_0)$ and $\gamma(\xi_1)$ as
\begin{equation}
	S(\xi_0,\xi_1)=\int_0^T|z'(t)|\sqrt{V_I(z(t))}dt.
\end{equation} 
Under suitable conditions,  the function $S$ is differentiable (at least in a neighbourhood of the point $(\bx,\bx)$), and 
\begin{equation}\label{eq:derS}
	\partial_{\xi_0}S(\xi_0,\xi_1) = -\sqrt{V_I(\gamma(\xi_0))}\frac{z'(0)}{|z'(0)|}\cdot\dot\gamma(\xi_0)
	\text{ and }
	\partial_{\xi_1}S(\xi_0,\xi_1) = \sqrt{V_I(\gamma(\xi_1))}\frac{z'(T)}{|z'(T)|}\cdot\dot\gamma(\xi_1).
\end{equation}
These relations represent the building blocks of the implicit function argument. 
}

Let us recall the first return map associated to our billiard, as introduced in \eqref{eq:intro_rit_rifl},
\[
F(\xi_0,\alpha_0) = (\xi_1,\alpha_1)
\]
defined on $[0,2\pi] \times \left(-\frac\pi2,\frac\pi2\right)$ into itself. Our aim now is to compute the Jacobian of $F$ in the fixed points $(0,0)$ and $(\pi,0)$ by means of Eqs. \eqref{eq:derS}.  Taking into account the definition of the angle $\alpha$ in the reflective case (i.e. the angle from the inner normal to the considered velocity) and the fact that the reflection law does not change the tangent component of the velocity, the above relations can be written as
\begin{equation}\label{eq:Phi}
	\begin{cases}
		\partial_{\xi_0}S(\xi_0,\xi_1)  -\sqrt{V_I(\gamma(\xi_0))}|\dot\gamma(\xi_0)| \sin(\alpha_0) = 0
		\\
		\partial_{\xi_1}S(\xi_0,\xi_1) + \sqrt{V_I(\gamma(\xi_1))}|\dot\gamma(\xi_1)| \sin(\alpha_1) = 0.
	\end{cases}
\end{equation}
We can hence say that $F(\xi_0,\alpha_0)=(\xi_1,\alpha_1)$ if and  only if the point $(\xi_0,\alpha_0,\xi_1,\alpha_1)$ is a zero of the vector function $\Phi =(\Phi_1,\Phi_2)$, whose components are the l.h.s. of two lines of \eqref{eq:Phi} .\\
Let us now consider a central configuration $\bar{\xi}$, corresponding to a fixed point $(\bar \xi, 0)$ for the map $F$, and the Jacobians of the map $\Phi$ at $(\bar \xi, 0,\bar \xi, 0)$ with respect to the variables $(\xi_0,\alpha_0)$, $(\xi_1, \alpha_1)$, which are given by
\[
D_{(\xi_0,\alpha_0)}\Phi(\bar \xi, 0,\bar \xi, 0) = 
\begin{pmatrix}
	\partial^2_{\xi_0\xi_0} S & -\sqrt{V_I}|\dot \gamma| \\
	\partial^2_{\xi_0\xi_1} S & 0
\end{pmatrix}
\]
and  
\[
D_{(\xi_1,\alpha_1)}\Phi(\bar \xi, 0,\bar \xi, 0) = 
\begin{pmatrix}
	\partial^2_{\xi_0\xi_1} S & 0 \\
	\partial^2_{\xi_1\xi_1} S & \sqrt{V_I}|\dot \gamma|
\end{pmatrix}
\]
where the second partial derivatives of $S$ are computed in $(\bar \xi,\bar\xi)$, $\sqrt{V_I} = \sqrt{V_I(\gamma(\bar \xi))}$ and $\dot \gamma = \dot \gamma (\bar \xi)$. \\
The computation of the second derivatives of $S$ is here omitted, but derives straightforwardly from Eq. (5.19) in \cite{deblasiterraciniellissi} and actually it turns out that
\begin{equation}\label{eq:der2S}
\begin{aligned}
	\partial^2_{\xi_0\xi_1} S = \frac{|\dot\gamma|^2}{4|\gamma|^2} \frac{\mu}{\sqrt{V_I}},
	\quad 
	\partial^2_{\xi_0\xi_0} S = \partial^2_{\xi_1\xi_1} S = -\partial^2_{\xi_0\xi_1} S +\varepsilon,
\end{aligned}
\end{equation}
where
\[
\varepsilon = \frac{|\dot\gamma|^2}{2|\gamma|^2} \frac{\mu+h_I|\gamma|}{\sqrt{V_I}} -2 \sqrt{|\gamma|} \sqrt{V_I} (1,0) \cdot \ddot \phi_-(\bar \xi).
\]
The formal definition of $\phi_-$ (which is a suitable transformation of the boundary needed to apply regularisation arguments) can be found in \cite[Eq. (5.17)]{deblasiterraciniellissi}. \\
It is then easy to show that $D_{(\xi_1,\alpha_1)}\Phi(\bar \xi, 0,\bar \xi, 0)$ is invertible and, by the implicit function theorem we can compute 
\[
DF(\bar \xi, 0) = -[D_{(\xi_1,\alpha_1)}\Phi(\bar \xi, 0,\bar \xi, 0)]^{-1}\, D_{(\xi_0,\alpha_0)}\Phi(\bar \xi, 0,\bar \xi, 0)
\]
that is
\[
DF(\bar \xi, 0) = 
\begin{pmatrix}
	\displaystyle
-\frac{\partial^2_{\xi_0\xi_0} S }{\partial^2_{\xi_0\xi_1} S } & \displaystyle \frac{\sqrt{V_I}|\dot \gamma|}{\partial^2_{\xi_0\xi_1} S } \\[5mm]
\displaystyle
\frac{\partial^2_{\xi_0\xi_0}S\  \partial^2_{\xi_1\xi_1} S - (\partial^2_{\xi_0\xi_1} S)^2}{\partial^2_{\xi_0\xi_1} S \sqrt{V_I}|\dot \gamma|} & \displaystyle - \frac{\partial^2_{\xi_1\xi_1} S }{\partial^2_{\xi_0\xi_1} S }
\end{pmatrix}
\]
The linear stability of $(\bx, 0)$ as fixed point of $F$ depends on the eigenvalues of  $DF(\bx,0)$, that is, on the sign of the discriminant $\Delta$ of the characteristic polynomial associated to the matrix. In particular, since the determinant of the matrix is $1$, whenever $\Delta$ is positive, $DF(\bx,0)$ has two eigenvalues $\lambda_1,\lambda_2=\lambda_1^{-1}\in \R$, and then $(\bx, 0)$ is a saddle point. If, on the contrary, $\Delta<0$, then the eigenvalues are complex, conjugated and unitary, and hence $(\bx,0)$ is a centre. From direct computations and using \eqref{eq:der2S}, one has 
\begin{equation}\label{eq:Delta}
	\Delta=\left(tr\left(DF(\bx,0)\right)^2-4\right)=\frac{4}{(\partial^2_{\xi_0\xi_1} S)^2}\varepsilon\left(\varepsilon-2\partial^2_{\xi_0\xi_1} S \right); 
\end{equation}
hence the sign of $\Delta$ depends on the quantity $\varepsilon(\varepsilon-\partial^2_{\xi_0\xi_1} S )$, which need to be computed separately in the case $\bx=0$ and $\bx=\pi$.  \\
Let us start by the case $\bx=0$: with the same techniques used in \cite{deblasiterraciniellissi} and taking into account the explicit expression of $\gamma$, one has 
\begin{equation}
	|\gamma|=1+x_0, \quad |\dot\gamma|=b, \quad \ddot\phi_-(0)=\frac{1}{2\sqrt{1+x_0}}\left(1-\frac{b^2}{2(1+x_0)}\right)(1,0), 
\end{equation}
from which one has 
\begin{equation}
	\begin{aligned}
	&\varepsilon_{|_{\bx=0}}=\sqrt{V_I(\gamma(0))}\left(\frac{b^2}{1+x_0}-1\right)\\
	&\varepsilon_{|_{\bx=0}}-2\partial^2_{\xi_0\xi_1} S(0,0)=\sqrt{V_I(\gamma(0))}\left(\frac{b^2}{1+x_0}-1-\frac{b^2\mu}{2(1+x_0)^2}\frac{1}{V_I(\gamma(0))}\right). 
	\end{aligned}
\end{equation}
From direct computations, it is possible to prove that, for every $0<b<1, h_I>0, \mu>0$ and $x_0>0$, in the case $\bx=0$ the quantity $\varepsilon(\varepsilon-\partial^2_{\xi_0\xi_1} S )$ is always positive: this means that the homothetic trajectory in the direction of $\gamma(0)=(1,0)$ is always unstable when the ellipse's centre is on the $x$-axis.\\
As for the case $\bx=\pi$, one has 
\begin{equation}
	\varepsilon_{|_{\bx=\pi}}(\varepsilon_{|_{\bx=\pi}}-\partial^2_{\xi_0\xi_1} S(	\pi,	\pi) )=V_I(\gamma(\pi))\left(\frac{b^2}{1-x_0}-1\right)\left(\frac{b^2}{1-x_0}-1-\frac{b^2\mu}{2(1-x_0)^2}\frac{1}{V_I(\gamma(\pi))}\right).  
\end{equation} 
Unlike the case $\bar\xi=0$, the above quantity could change sign depending on $x_0, h_I, b$ and $\mu$, leading to bifurcation phenomena. We will provide numerical evidences of this, characterizing such bifurcation as a pitchfork. 
\begin{theorem}\label{thm:DeltaKep}
	Let $\gamma(\xi)=(\cos\xi+x_0, b \sin\xi)$, $x_0\in[0,1)$, $b\in(0,1)$ the parameterisation of an ellipse with eccentricity $e=\sqrt{1-b^2}$, semimajor axis equal to $1$ and the centre lying on the positive $x-axis$. Then the central configuration $\bx=0$ is always unstable, while the stability of $\bx=\pi$ depends on the sign of the quantity
	\begin{equation}\label{eq:Cpi}
		C_{\pi}=\left(\frac{b^2}{1-x_0}-1\right)\left(\frac{b^2}{1-x_0}-1-\frac{b^2\mu}{2(1-x_0)^2}\frac{1}{V_I(\gamma(\pi))}\right), 
	\end{equation}
in the sense that
\begin{itemize}
	\item if $C_\pi>0$, then $\bx=\pi$ is unstable; 
	\item if $C_\pi<0$, then $\bx=\pi$ is stable. 
\end{itemize}
The stability of $0$ and $\pi$ swaps when $x_0\in(-1,0]$. 
\end{theorem}
\begin{remark}
	Up to Eq.\eqref{eq:Delta}, the method described is completely general, and can be considered valid to compute the linear stability of any central configuration in a general domain. The computation of the explicit expression of the partial second derivatives of $S$ depends instead on $\gamma$, as well as on $\bx$. 
\end{remark}

\begin{remark}\label{rem:prop_cerchio}
	Theorem \ref{thm:DeltaKep} trivially extends to the circular non centred case $b=1$, $x_0 \neq 0$, which guarantees the existence of two antipodal non-degenerate central configurations.
\end{remark}

\section{Numerical simulations}\label{sec:numerical}

In this section we pass from a purely analytical analysis of our model to a numerical investigation. We will provide several examples of phase portraits both in the reflective and refractive case taking into account different domains and values of the physical parameters. The Poincar\'e sections are obtained by mean of $Mathematica^\copyright$\, software. We highlight that the aim of this analysis is dual: on one hand, a comparison with the expected theoretical results from Theorem \ref{thm:DeltaKep} and its analogous in \cite{deblasiterraciniellissi} will corroborate the validity of our routines. Secondly, it allows the exploration of cases still not covered by the analytical results of Theorems \ref{thm:main_lavoro} and \ref{thm:primo_teorema}.

\subsection{Evidences of bifurcations} 
\label{subsec:bif} Here, we propose some simulations that validate our numerical routines and display the presence of bifurcation phenomena both in the reflective and the refractive case. 

The first example we propose is a focused elliptic reflective Keplerian billiard:  this system is integrable and admits two homothetic equilibria in $(0,0)$ and $(\pi,0)$, for any value of the physical parameters. Figure \ref{fig:rifl_integ} shows the phase portrait for two different energy regimes (our choice is to fix the ellipse and $\mu$ and change $h_I$, although of course one can chose to fix the energy and move another quantity); on the left, one can check that the quantity $C_\pi$ in Eq. \eqref{eq:Cpi} is negative, and in fact we observe that $(\pi,0)$ is a centre. Conversely, on the right,  $C_\pi>0$ and $(\pi,0)$ is a saddle. 

\begin{figure}[h!]
	\includegraphics[width=0.4\textwidth]{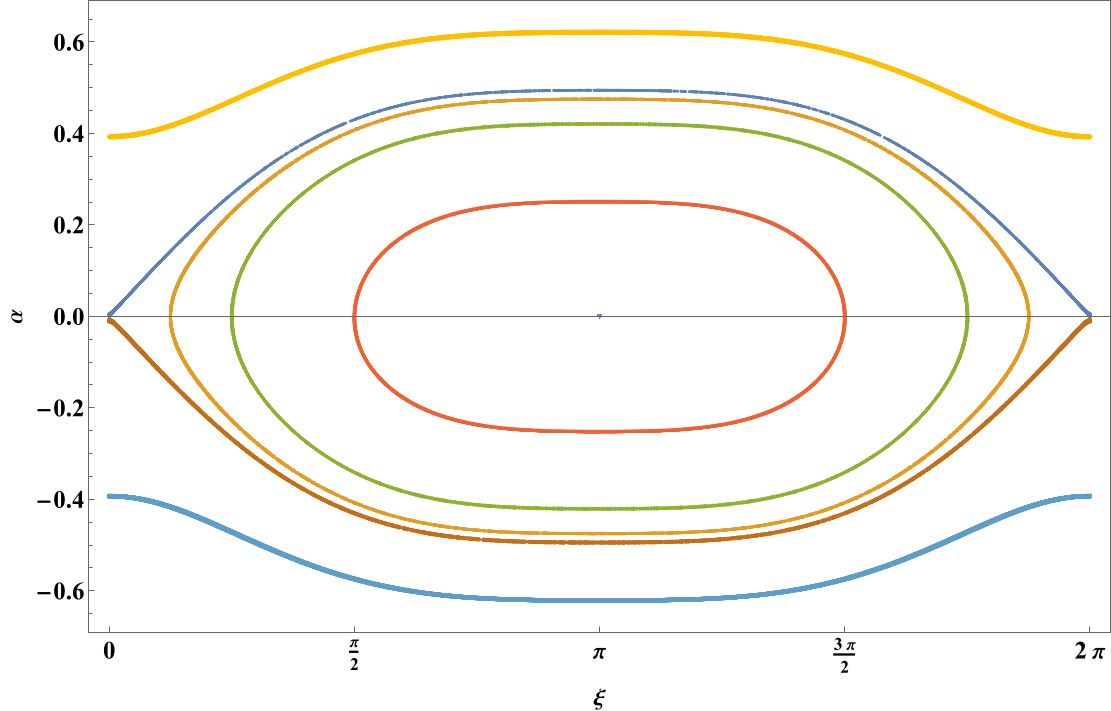} \qquad 
	\includegraphics[width=0.4\textwidth]{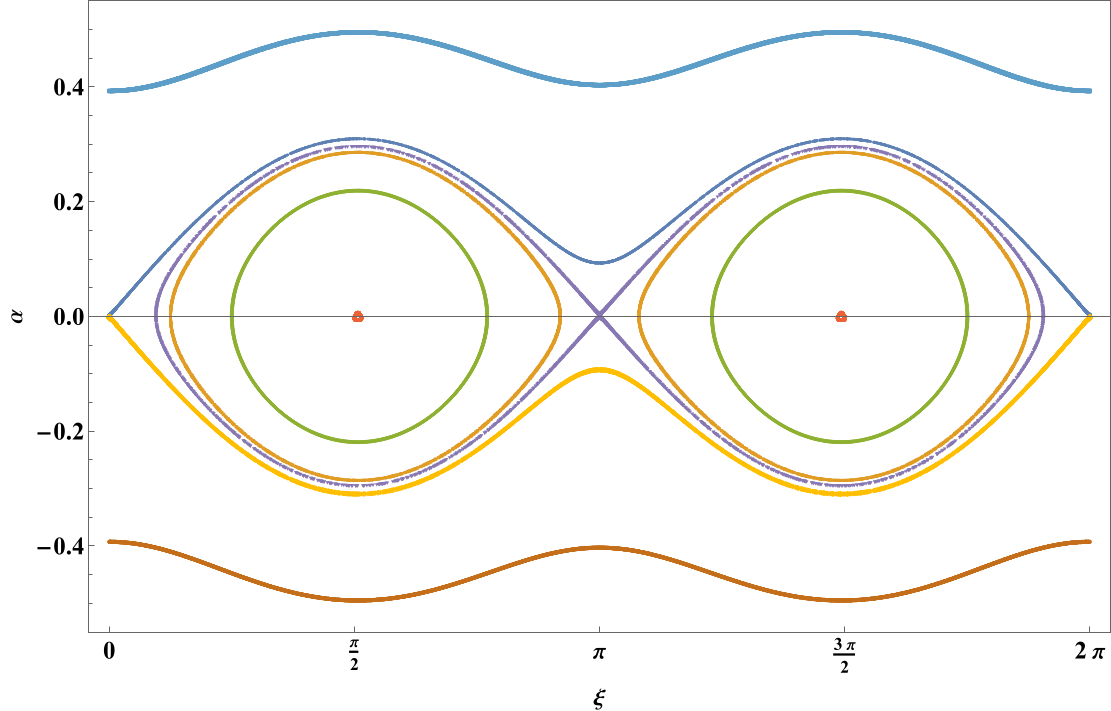}
	\caption{Phase portrait of an elliptic Keplerian reflective billiard with the mass at a focus (integrable case). The physical parameters are: $x_0=e=0.3$, $\mu=2$, $h_I=3$ (left) and $h_I=150$ (right).}
	\label{fig:rifl_integ}
\end{figure}

\noindent Let us observe that, as far as the mass parameter, the centre and the eccentricity of the ellipse are fixed, the quantity $C_\pi$ just depends on $h_I$ and it is strictly increasing. Furthermore, there exists a transition value, $\tilde h$, such that when $h_I < \tilde h$, then $(\pi,0)$ is a centre, while when $h_I > \tilde h$, then $(\pi,0)$ is a saddle (see Figure \ref{fig:transizione_rifl}). 
\begin{figure}[h!]
	\includegraphics[width=0.4\textwidth]{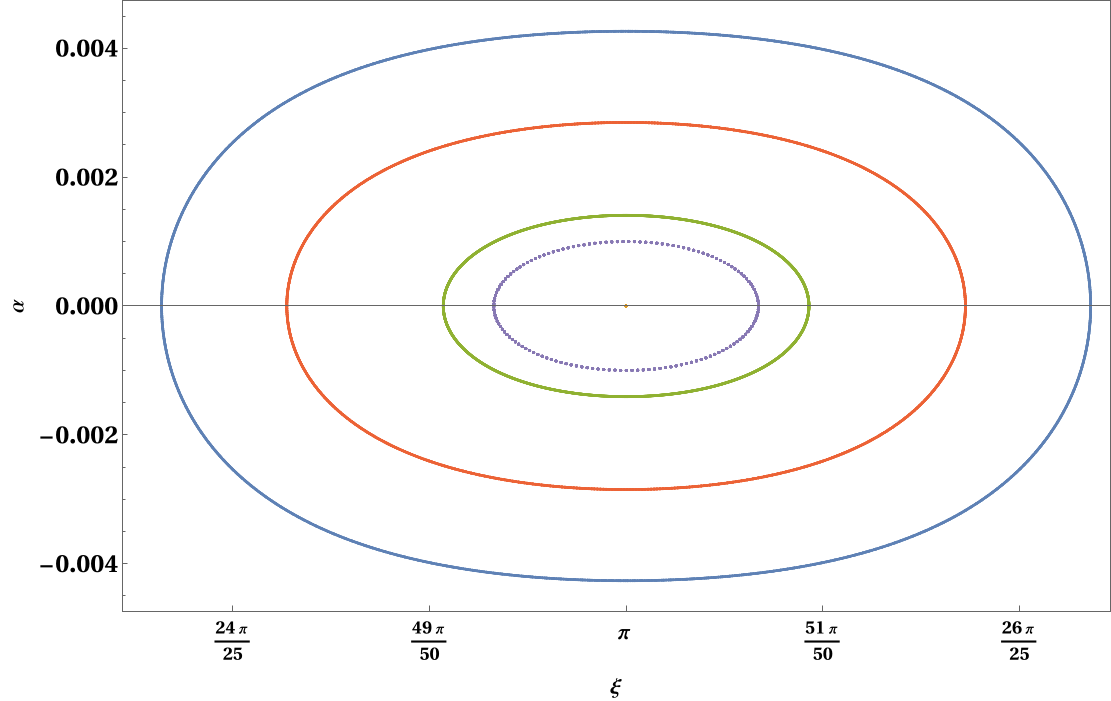} \qquad 
	\includegraphics[width=0.4\textwidth]{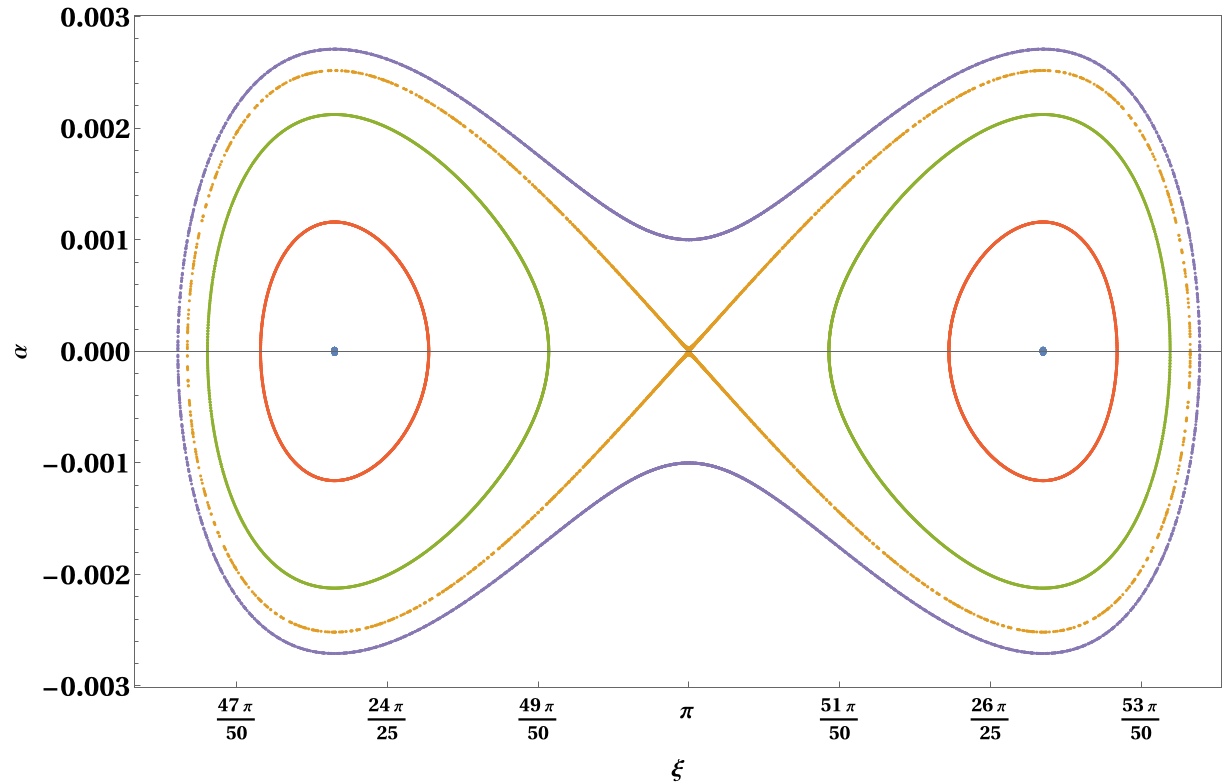}\\
	\includegraphics[width=0.4\textwidth]{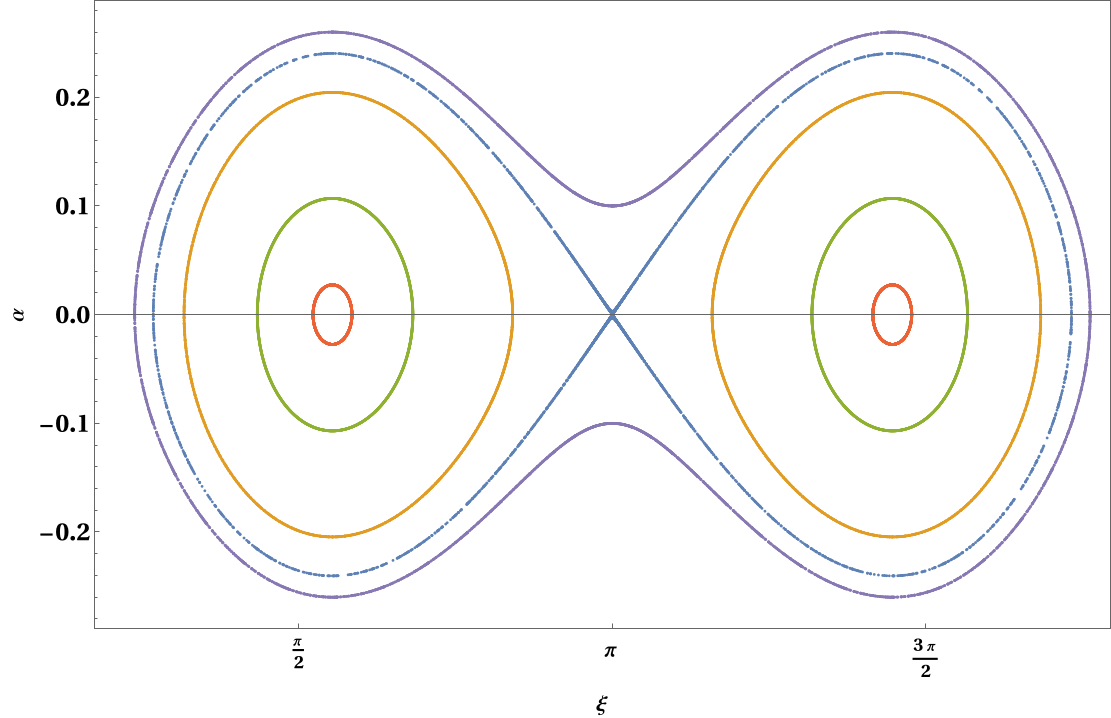}\qquad
	\includegraphics[width=0.4\textwidth]{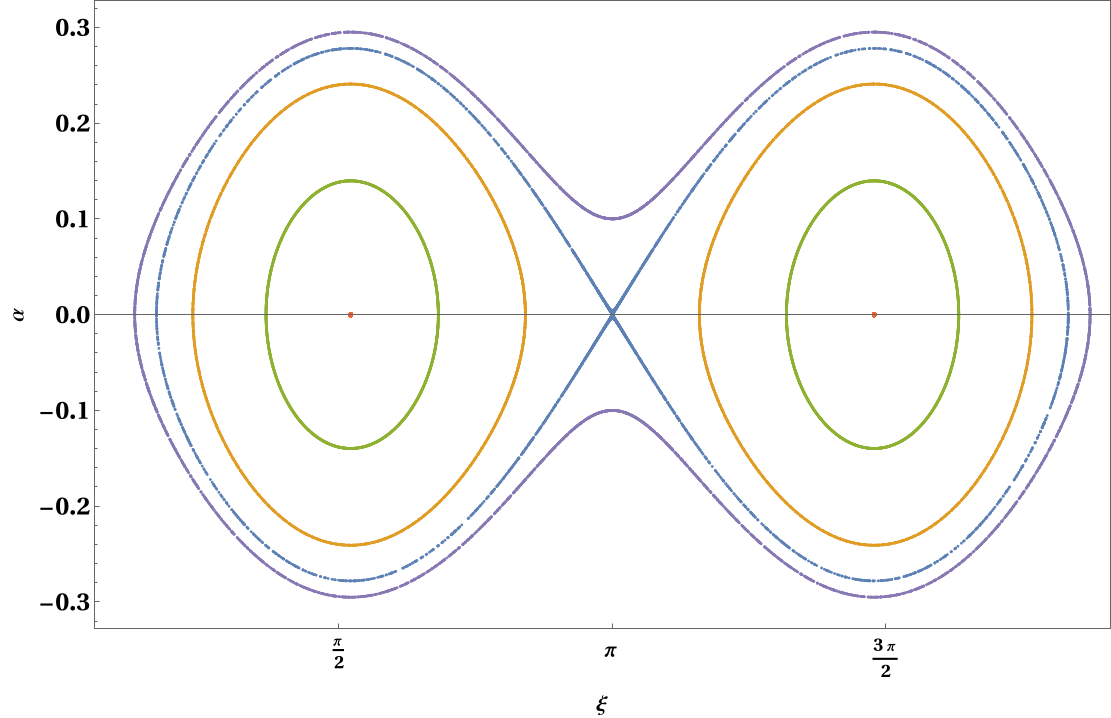}
	\caption{Phase portrait in a neighbourhood of $(\pi,0)$ for the elliptic Keplerian reflective billiard with physical parameters: $x_0=e=0.3$, $\mu=2$. With this choice, the bifurcation value is $\tilde h \approx 3.333$: in the first line we catch this transition and the inner energy is $h_I=3.3$ (left) and $h_I=3.7$ (right). We can observe a pitchfork bifurcation: the centre splits into  a saddle and a pair of 2-periodic points. In the second line we chose $h_I=20$ (left) and $h_I=50$ (right): here we can observe that the distance between the two periodic points increases, as they tend to converge to $(\pi/2,0)$ and $(3\pi/2, 0)$. We propose Table \ref{tab:brake} for a more detailed description of this behaviour.}
	\label{fig:transizione_rifl}
\end{figure}

\noindent Figures \ref{fig:rifl_integ} and \ref{fig:transizione_rifl} also show that a pitchfork bifurcation occurs when $h_I$ crosses the value $\tilde h$: the centre at $(\pi,0)$ becomes a saddle and a pair of two-periodic points arises in its neighbourhood. Such two-periodic points still belong to the horizontal axis, $\alpha=0$: let us call them $(\xi_{brake},0)$ and, by symmetry, $(2\pi-\xi_{brake},0)$. Translating them to the two-dimensional dynamics, it turns out that they correspond to a brake trajectory (see Figure \ref{fig:rifl_brake}), which bounces into itself at every intersection with the ellipses.
\begin{figure}[h!]
	\includegraphics[width=0.25\textwidth]{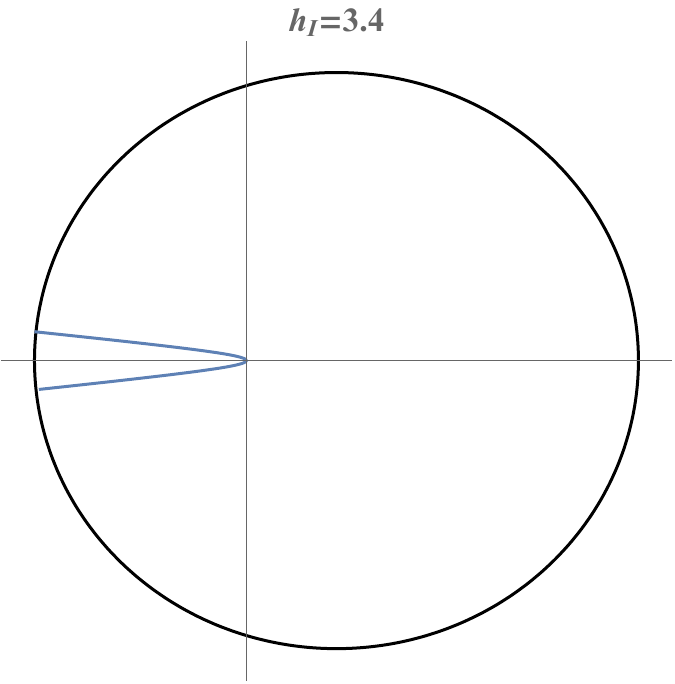} \qquad 
	\includegraphics[width=0.25\textwidth]{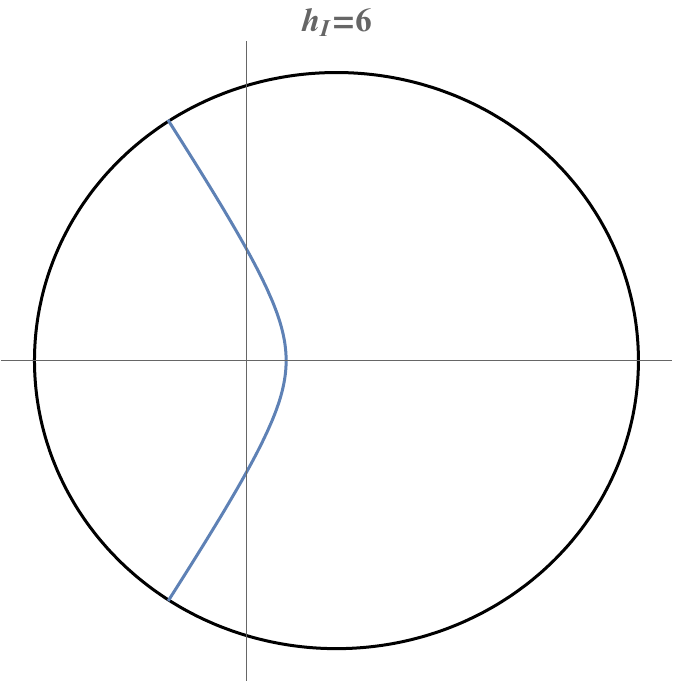}\qquad
	\includegraphics[width=0.25\textwidth]{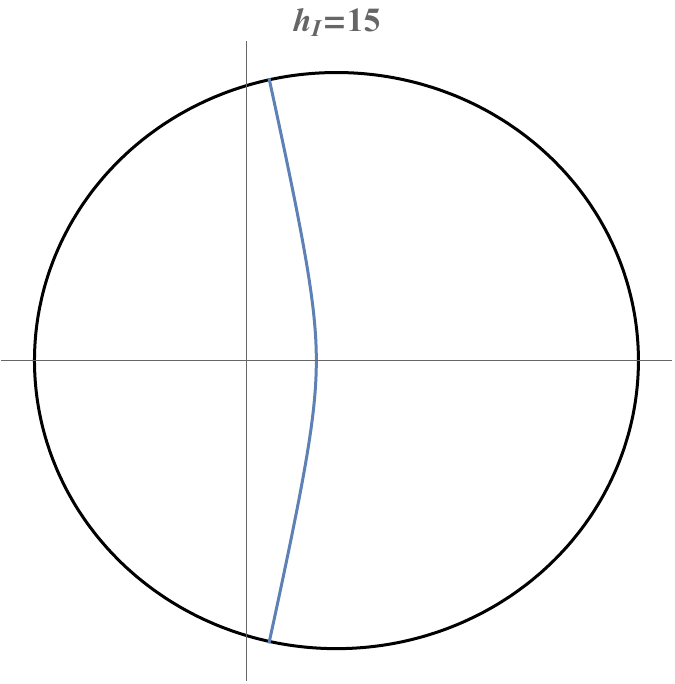}
	\caption{Brake orbits corresponding to the two-periodic points rising from the pitchfork bifurcation of $(\pi,0)$. When $h_I$ increases, the orbit becomes more and more straight and tends to the vertical two-periodic trajectory of a classical Birkhoff billiard.}
	\label{fig:rifl_brake}
\end{figure}
Table \ref{tab:brake} shows the parameter $\xi_{brake}$  starting from values of $h_I$ very close to $\tilde h$, and their distance from $\pi/2$. It is evident how, for increasing values of $h_I$, such point tends to get closer and closer to $\pi/2$; a specular behaviour can be showed for the convergence of the right brake point to $(3\pi/2,0)$. {We highlight that, from a rigorous analytical point of view, the analysis of such brake trajectories can be carried on by means of the \textit{free fall method}, already used in a similar case in \cite{deblasiterraciniellissi} for the refractive case}. 

\begin{table}[h!]
	\begin{tabular}{|c|c|c||c|c|c|}
		\hline
		\rule[-1ex]{0pt}{2.5ex} $h_I$ & $\xi_{brake}$ & $|\xi_{brake}-\pi/2|$ & $h_I$ & $\xi_{brake}$ & $|\xi_{brake}-\pi/2|$ \\	
		\hline\hline
		\rule[-1ex]{0pt}{2.5ex} 3.35 & 3.0418 & 1.471 & 30 & 1.6821 & 0.1113\\
		\hline
		\rule[-1ex]{0pt}{2.5ex} 3.5 & 2.8318 & 1.2610 & 50 & 1.6375 & 0.0667 \\
		\hline
		\rule[-1ex]{0pt}{2.5ex} 4 & 2.5559 & 0.9851 &70 & 1.6184 & 0.0476 \\
		\hline
		\rule[-1ex]{0pt}{2.5ex} 6 & 2.1598 & 0.5890 & 100 & 1.6041 & 0.0333 \\
		\hline
		\rule[-1ex]{0pt}{2.5ex} 10 & 1.9106 & 0.3398 & 150 & 1.5930 &0.0222  \\
		\hline
		\rule[-1ex]{0pt}{2.5ex} 15 & 1.7949 & 0.2241 & 250 & 1.5841 & 0.0222 \\
		\hline
		\rule[-1ex]{0pt}{2.5ex} 20 & 1.7382 & 0.1674 & 500 & 1.5775 & 0.0067 \\
		\hline
	\end{tabular}
	\caption{Parameter $\xi_{brake}$ of the left 2-periodic point of a brake orbit for the integrable elliptic reflective case when $(0,\pi)$ is a saddle, with physical values $x_0=e=0$, $\mu=2$, and for different values of the energy. It is evident how $\xi_{brake} \to \pi/2$ as the energy increases. A geometric representation of the behaviour of the corresponding brake trajectories in some cases is depicted in Figure \ref{fig:rifl_brake}. }
	\label{tab:brake}
\end{table}

Our second example regards the refractive case and consider a non-centred circle; we highlight that, from a computation point of view, the routines used in this paper are different from the ones proposed in \cite{deblasiterraciniellissi} which took into account just centred ellipses.  
Here the comparison has to be made with Theorem 1.1 in \cite{deblasiterraciniellissi}: Figure \ref{fig:xc04trans} shows a high coherence between numerical and theoretical results. We can notice that also in this case the transition from centre to saddle is associated with a pitchfork bifurcation, but here the two-periodic points do not correspond to a brake trajectory any more. This is due to the presence of an outer dynamics which moves the starting points of inner arcs. 

\begin{figure}[h!]
	\includegraphics[width=0.4\textwidth]{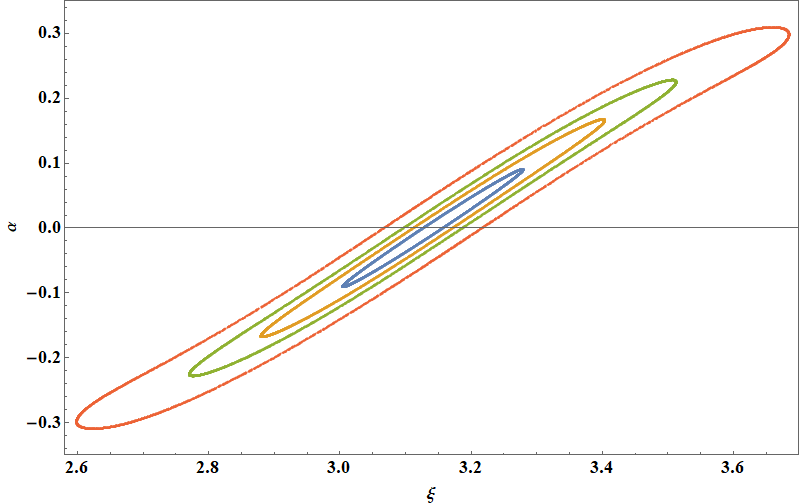} \qquad 
	\includegraphics[width=0.4\textwidth]{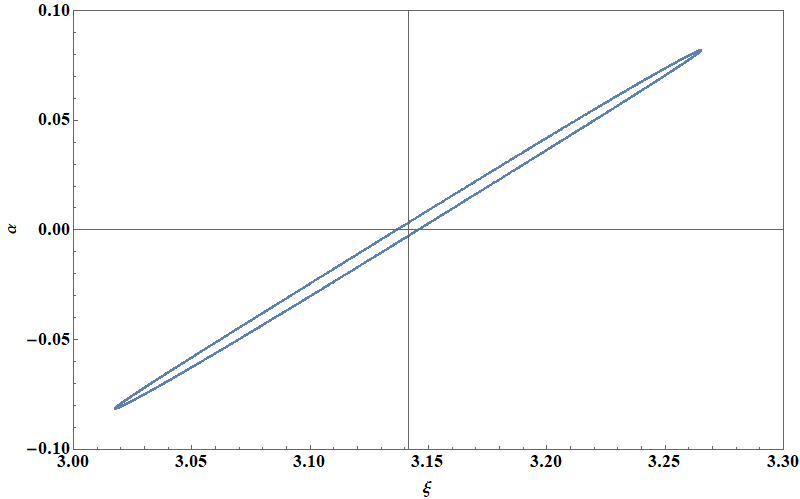} \\
	\includegraphics[width=0.4\textwidth]{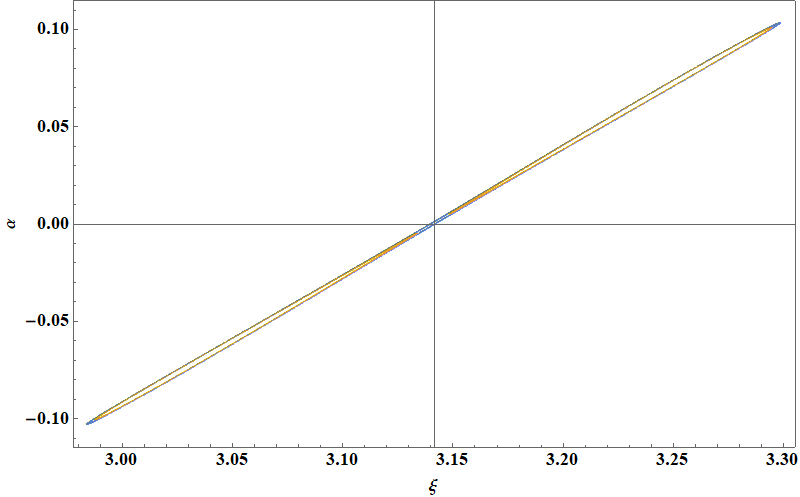} \qquad 
	\includegraphics[width=0.4\textwidth]{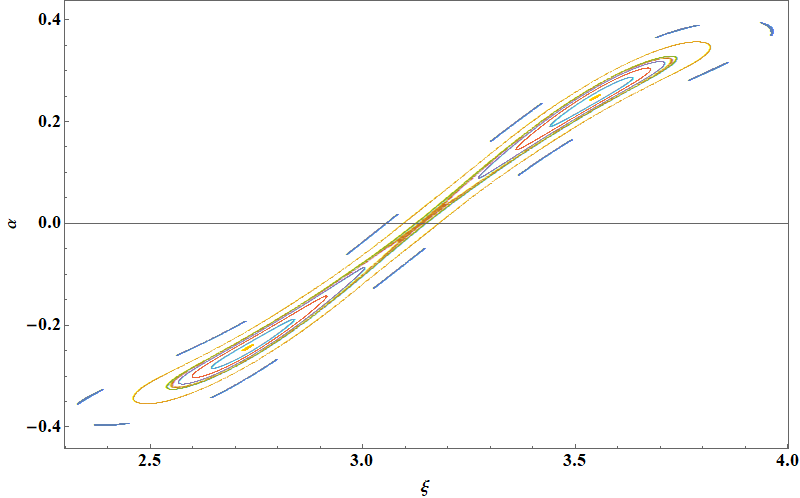} 
	\caption{Phase portrait in a neighbourhood of $(\pi,0)$ for the non-centred circular Keplerian refractive billiard with physical parameters: $x_0=0.4$, $h_E=9$, $\mu=2$, $\omega =1$. With this choice, the bifurcation value is $\tilde h \approx h_E +3.73987$: in the first line the values of the inner energy are $h_I=h_E +3$ and $h_I=h_E +3.69$, in the second one $h_I=h_E +3.79$ and $h_I=h_E +4.5$. It is evident the transition from centre to saddle.}	
	\label{fig:xc04trans}
\end{figure}

\subsection{Evidences of chaos}\label{subsec:num_chaos}
In this section we propose a numerical analysis in the non-admissible case, that means the situations not covered by Theorem \ref{thm:main_lavoro}. Let us outline that this theoretical result does not provide estimates on the energy levels that guarantee chaos: the simulations we propose here can be intended as a complement in this direction. \\
What we can notice is that, whenever integrability is not already guaranteed (focused ellipses in the reflective case and centred circles in both cases), chaotic behaviours appear, at least locally.

For the reflective model we propose simulations in two different domain shapes. The first one concerns a circular billiard whose centre is very close, although not coinciding, to the Keplerian mass. In Figure \ref{fig:cerchio non centrato} we focus our attention near the saddle in $(0,0)$ where we expect to observe the arising of chaos. Let us notice that  very local diffusive phenomena are noticeable already for low energies. 
\begin{figure}[h!]
	\includegraphics[width=0.4\textwidth]{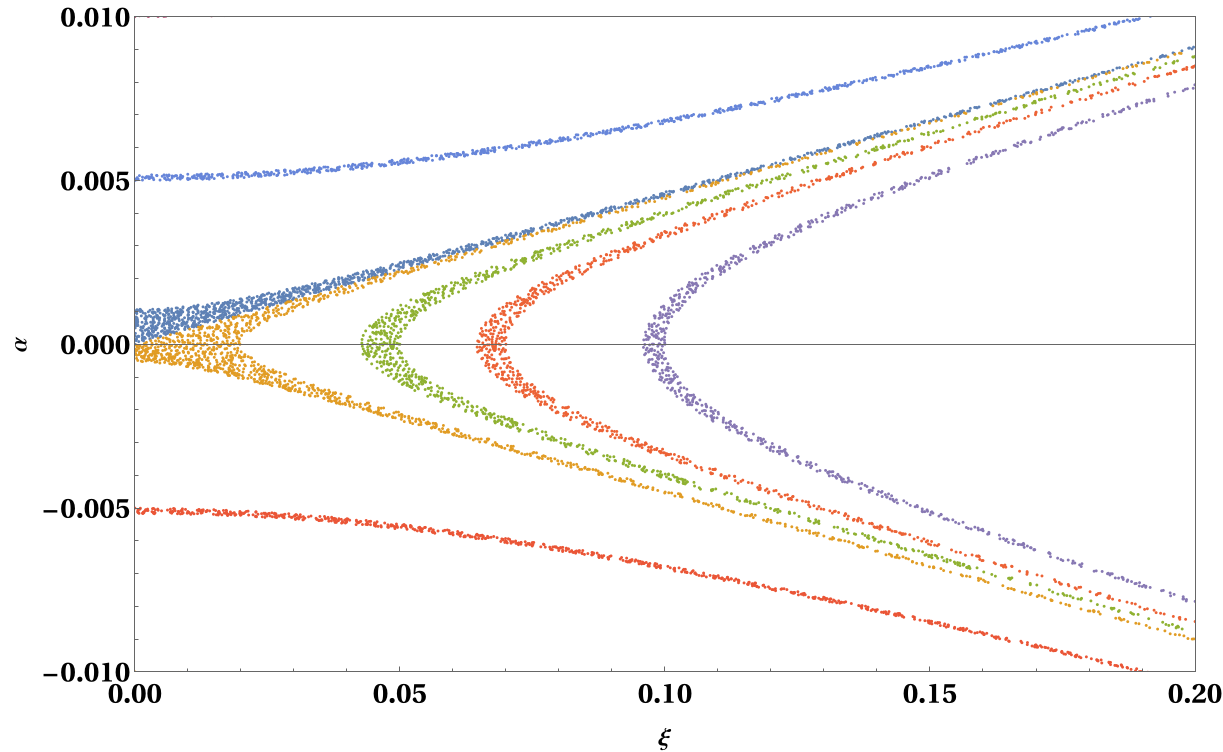} \qquad
	\includegraphics[width=0.4\textwidth]{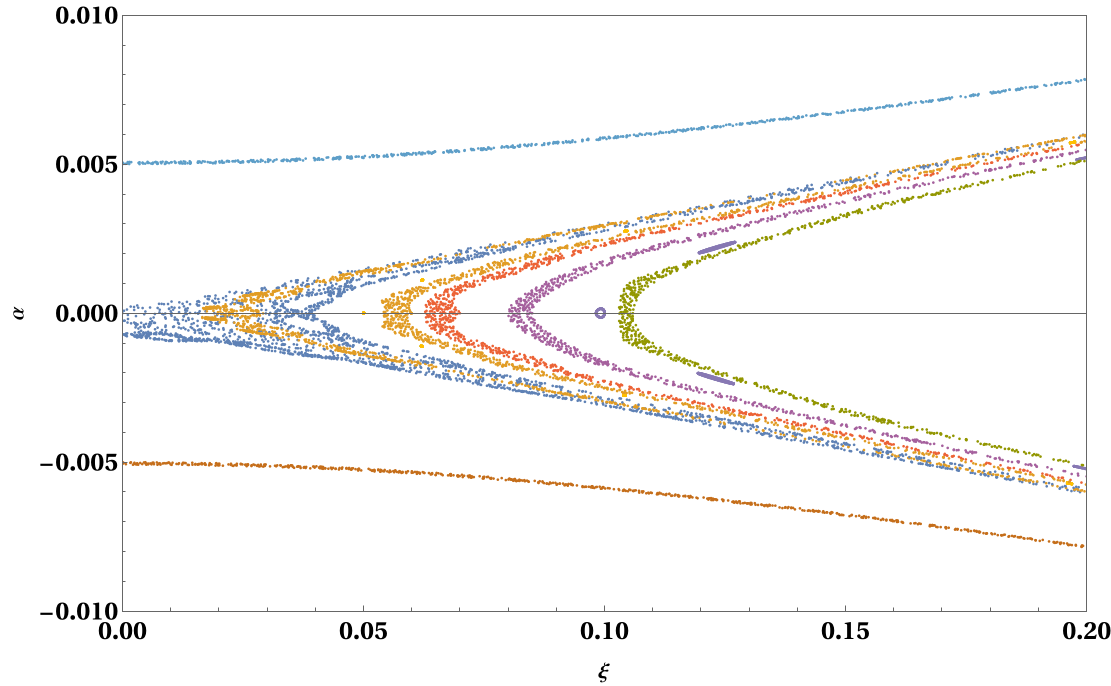} 
	\caption{Phase portrait in a neighbourhood of $(0,0)$ for the non-centred circular Keplerian reflective billiard with physical parameters: $x_0=0.01$, $\mu=2$, $h_I=3$ (left) and $h_I=10$ (right). We note the presence of diffusive orbits near the saddle point with an increasing complexity of the dynamic for higher values of $h_I$.}
	\label{fig:cerchio non centrato}
\end{figure}
As for non focused ellipses, we refer to Figures \ref{fig:ellisse non focused 01_x0.31},  \ref{fig:ellisse non focused sx}. It is particularly interesting that, moving slightly the centre from the focus, very evident chaotic orbits are present, even with low energies. Such diffusive orbits, initially limited near the saddles, cover a larger part of the phase space when the centre is taken further from the focus.  
\begin{figure}[h!]
	\includegraphics[width=0.4\textwidth]{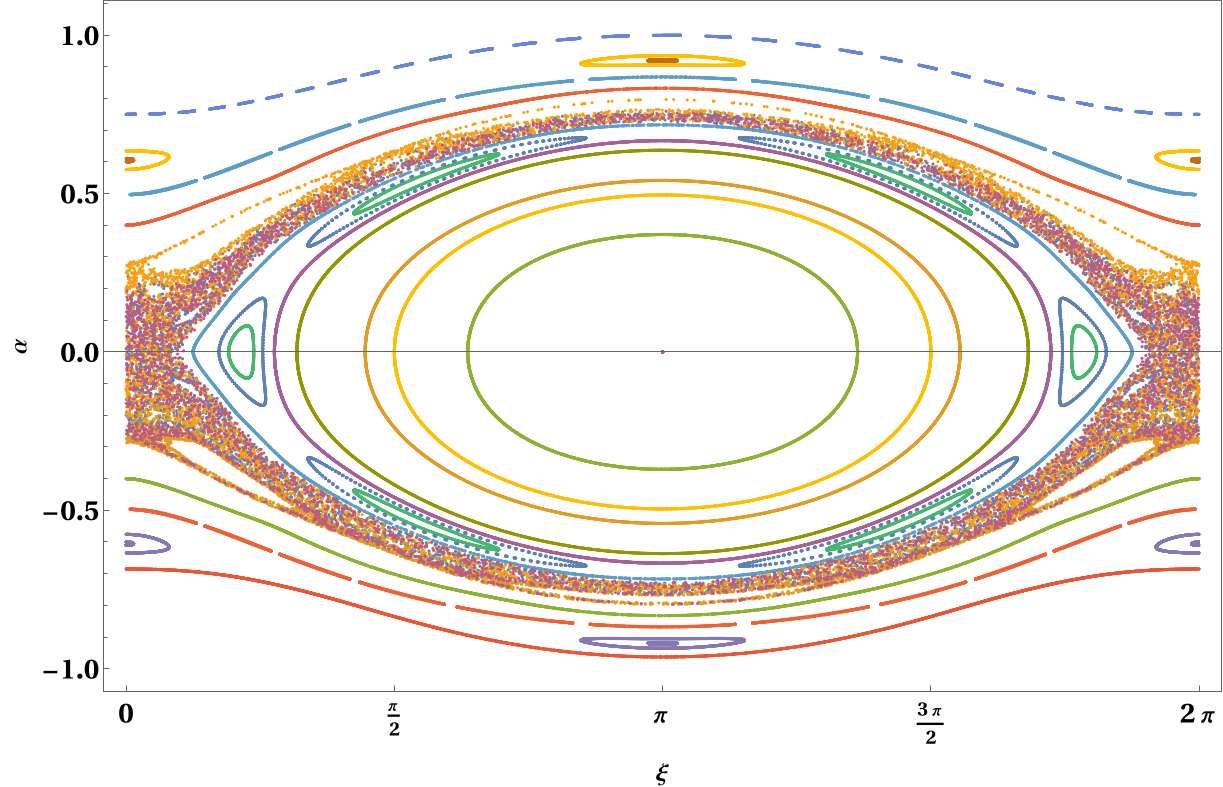} \qquad
	\includegraphics[width=0.4\textwidth]{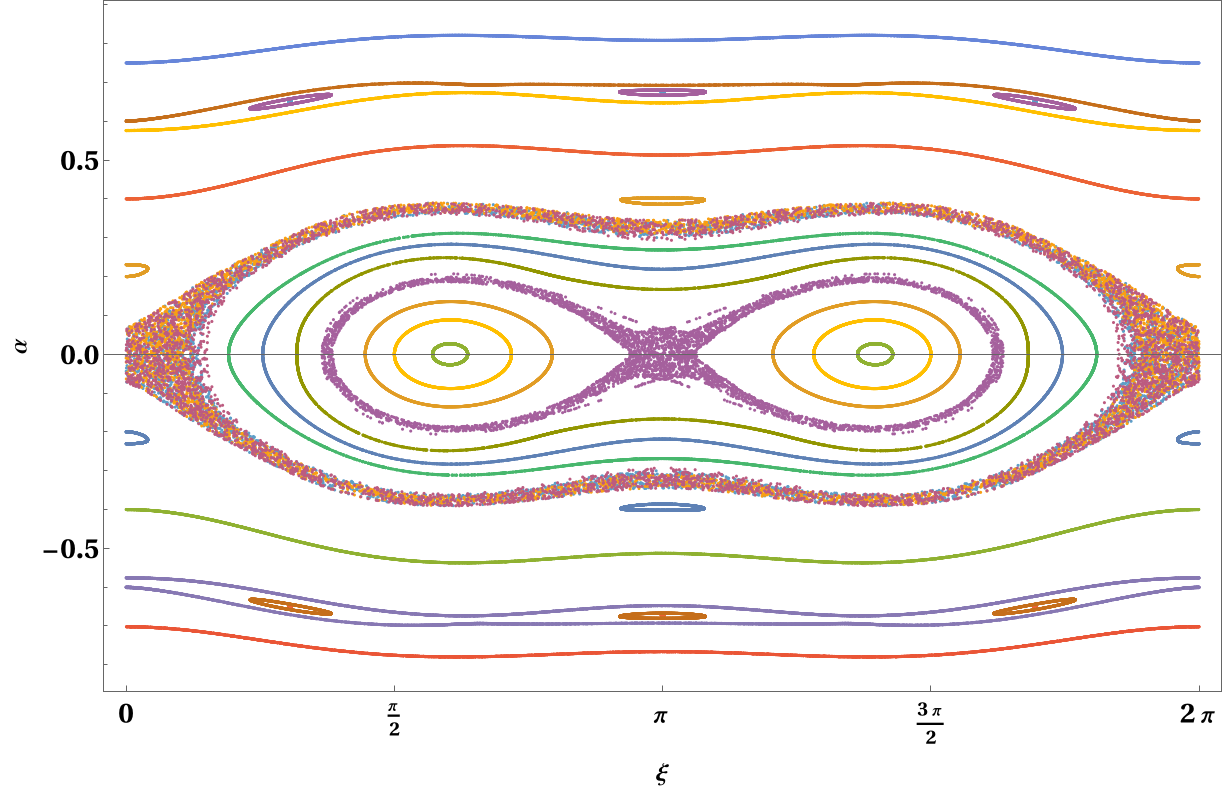} 
	\caption{Phase portrait for the non-focused elliptic Keplerian reflective billiard with physical parameters: $e=0.3$, $x_0=0.31$, $\mu=2$, $h_I=0.1$ (left) and $h_I=10$ (right). A marked chaotic behaviour is evident even for small deviations from the integrable case ($x_0=0.3$) and small inner energies. Compare with Figure \ref{fig:rifl_integ}.}
	\label{fig:ellisse non focused 01_x0.31}
\end{figure}
\begin{figure}[h!]
	\includegraphics[width=0.5\textwidth]{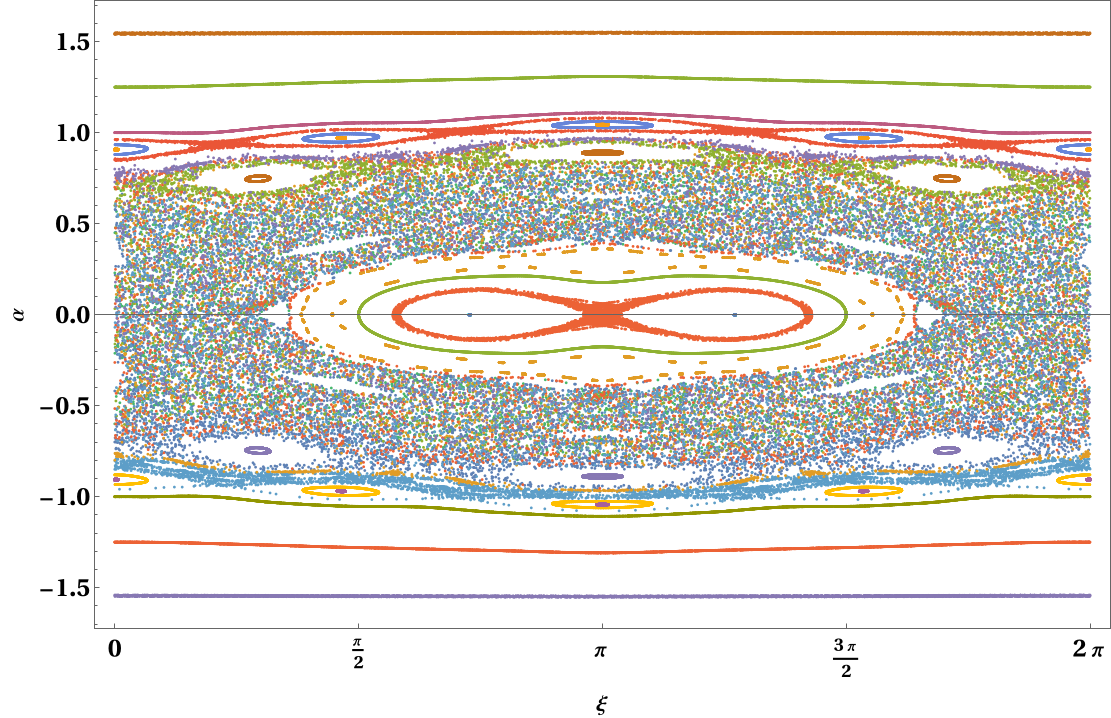} 
	\caption{Phase portrait for the non-focused elliptic Keplerian reflective billiard with physical parameters: $e=0.3$, $x_0=0.4$, $\mu=2$, $h_I=3$.}
	\label{fig:ellisse non focused sx}
\end{figure}

We now investigate the presence of chaos in the refractive case. Here we have an additional problem: the dynamics is not always well defined, due to the presence of the critical angle (see Section \ref{sec:intro}, Eq. \eqref{eq:crit_ang_intro}); as we will see in the following simulations, this will translate in the presence of invariant curves bounding the portion of the phase space in which the dynamics is well defined. Such curves are particularly interesting also because, as we will notice, very local chaotic phenomena can arise in their vicinities. \\
For the refractive case, we refer to Figures \ref{fig:xc002}, \ref{fig:xc01} and \ref{fig:xc04}, which show the behaviour of the first return map for non-centred circles whose centres have coordinates respectively $(0.02, 0)$, $(0.1,0)$ and $(0.4,0)$, and for different energy regimes. We highlight that in all the cases described the domains are non-admissible, and then Theorem \ref{thm:main_lavoro} does not guarantee the presence of chaotic behaviour; moreover, let us recall that the case gets further and further from the centred circle case, which is integrable and highly degenerate (see \cite{IreneSusNew}).  \\
In general, one can observe the presence of chaotic phenomena in almost every regime considered: they became more evident as the inner energy $h_I$ increases or the mass $\mu$ goes further from the centre. In Figure \ref{fig:xc002} the mass is extremely close to the circle's centre: for low energies no diffusive orbits are detectable; nevertheless, taking higher values of $h_I$ (see top right figure), one can start detecting some very local evidences of chaos, associated with the appearance of secondary periodic islands. This is clear when the energy increases (bottom row), where diffusive orbits around the saddle in $(0,0)$ are clearly visible. \\
Taking into account the case $x_0=0.1$ (Figure \ref{fig:xc01}), it is evident the presence of a dynamics which is more and more complex as $h_I$ increases, starting with the presence of secondary periodic island to evidently diffusive orbits. In particular, it is interesting the comparison between Figure \ref{fig:xc002} top right and Figure \ref{fig:xc01} bottom right: in principle, they both represent chaotic phenomena arising on the invariant trajectories bounding the region of the phase space where the refractive first return map is well defined; nevertheless, the dimension of libration islands and corresponding diffusive orbits differs by one order of magnitude. \\
The last example we propose is in Figure \ref{fig:xc04}, and corresponds to the case $x_0=0.4$: here, already for low inner energies, diffusive phenomena are evident. 
\begin{figure}[h]
\centering
\includegraphics[width=0.4\textwidth]{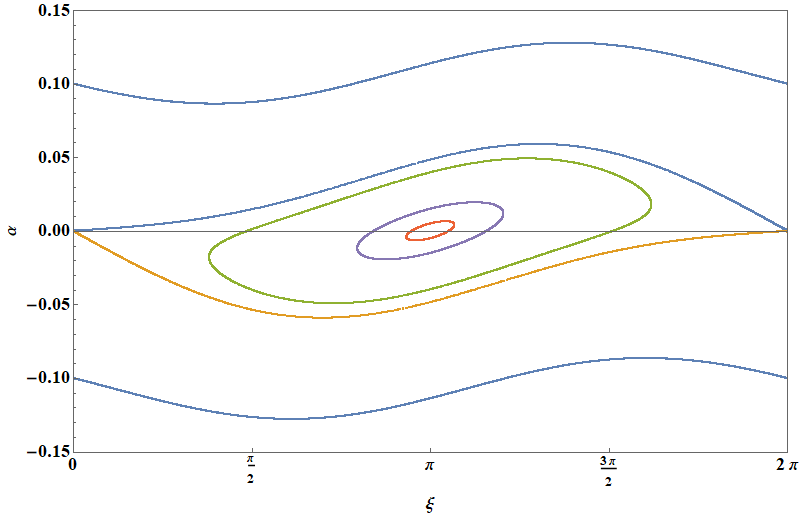} \qquad 
\includegraphics[width=0.4\textwidth]{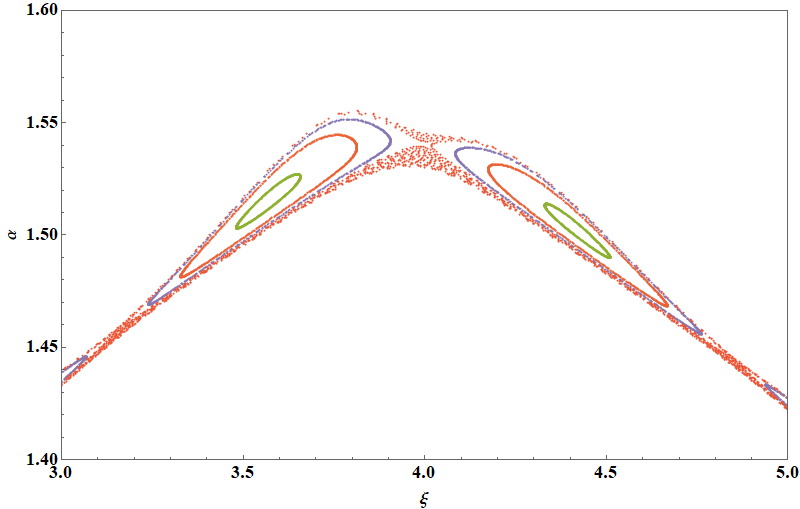} \\
\includegraphics[width=0.4\textwidth]{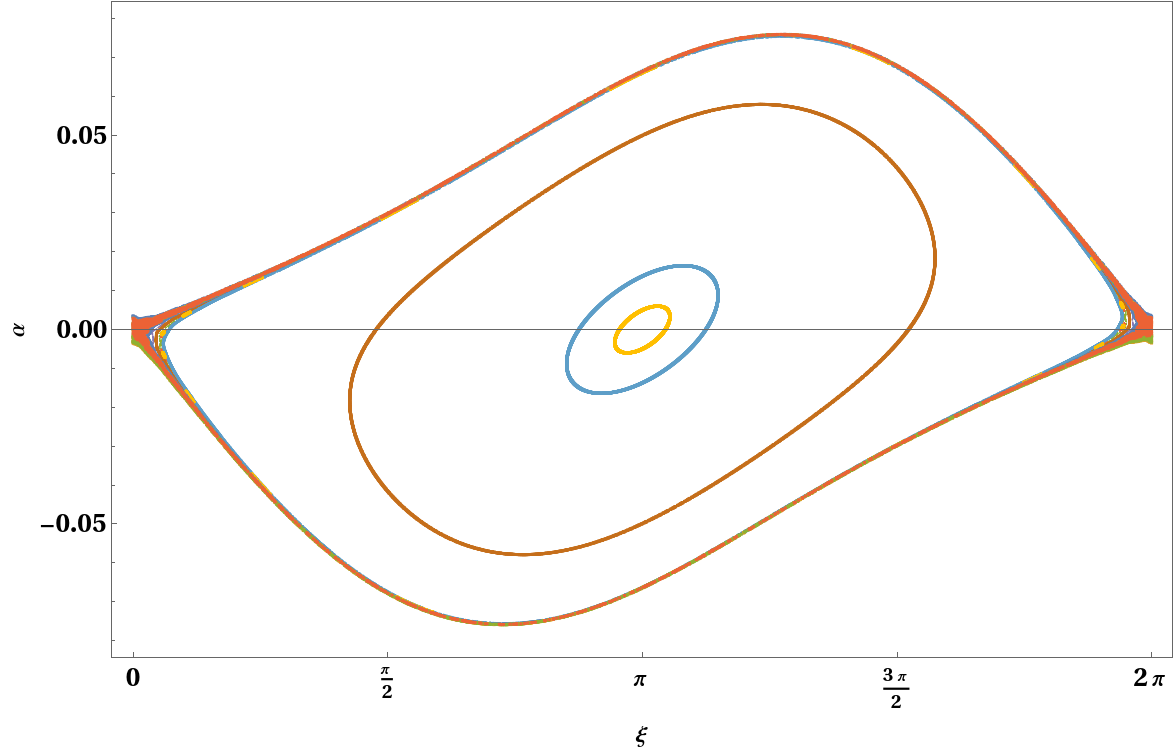} \qquad 
\includegraphics[width=0.4\textwidth]{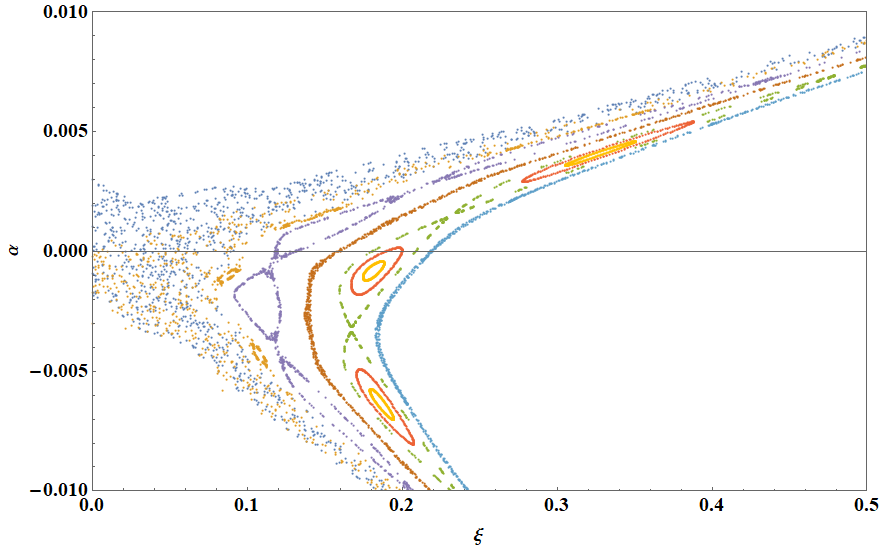}
\caption{Phase portrait for the non-centred circular Keplerian refractive billiard with fixed physical parameters $x_0=0.02$, $h_E=9$, $\mu=2$ and $\omega=1$. In the first line we have on the left a complete phase portrait when $h_I=h_E+3$, while on the right a detail showing an arising of very local diffusive orbits when $h_I=h_E+6$. 
In the second line $h_I=h_E+15$ and the chaos is evident also near the saddle in $(0,0)$.
}
\label{fig:xc002}
\end{figure}

\begin{figure}[h]
	\centering
	\includegraphics[width=0.4\textwidth]{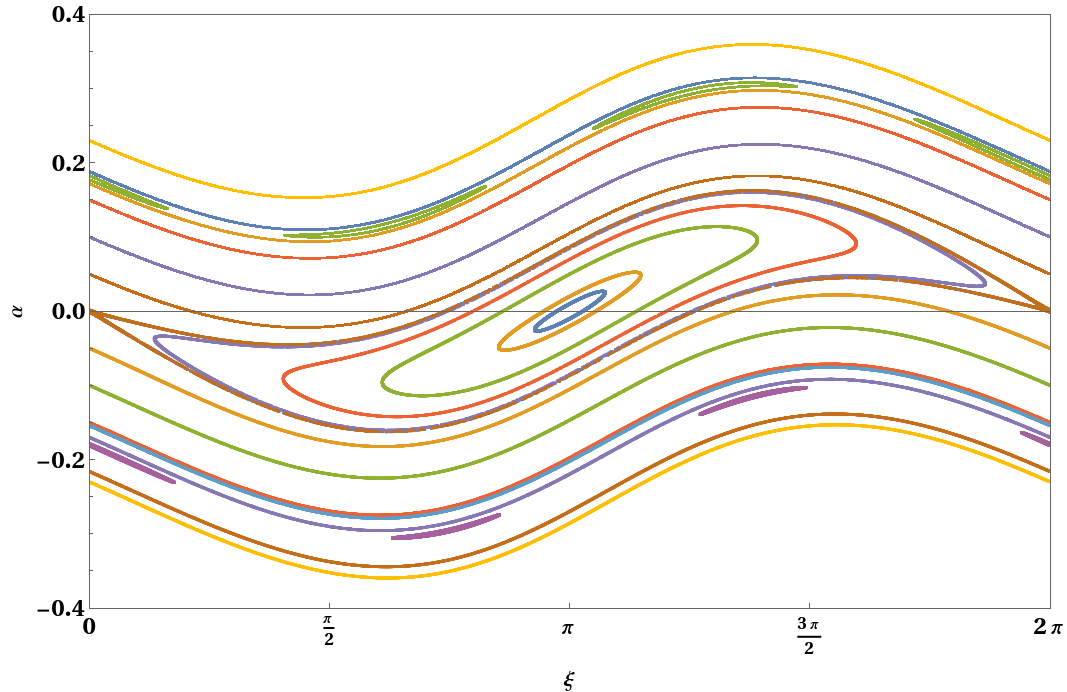} \qquad 
	\includegraphics[width=0.4\textwidth]{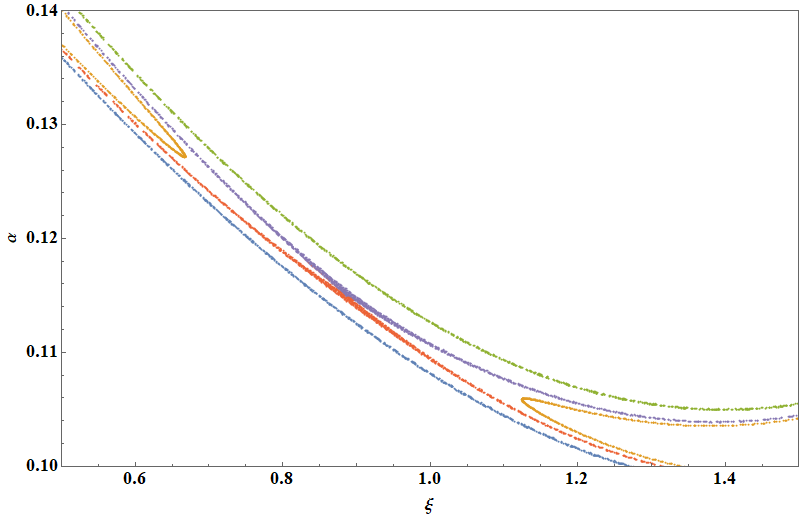} \\
	\includegraphics[width=0.4\textwidth]{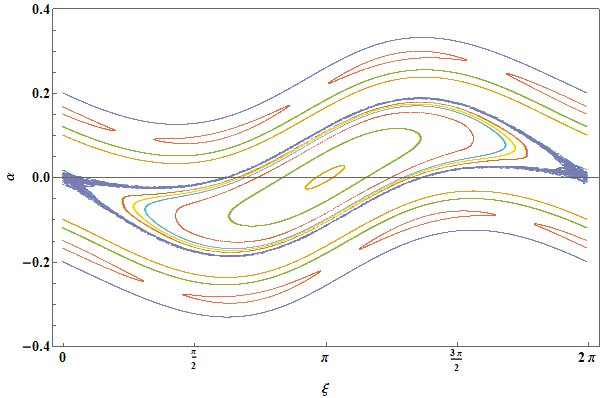} \qquad 
	\includegraphics[width=0.4\textwidth]{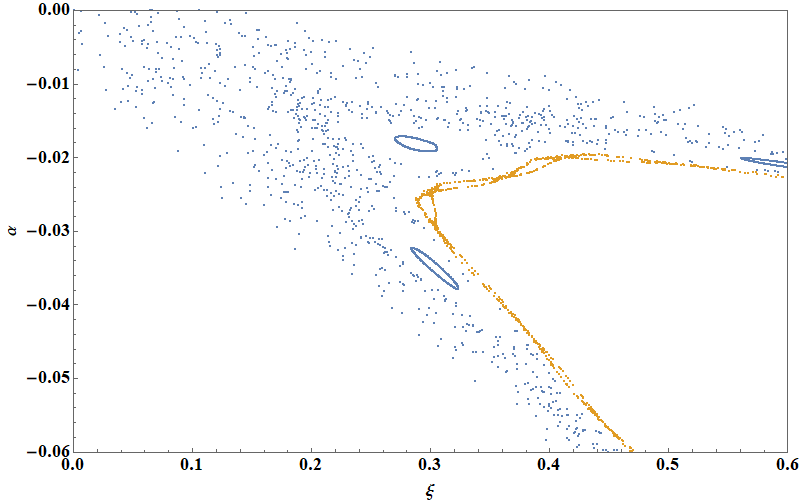} \\
	\includegraphics[width=0.4\textwidth]{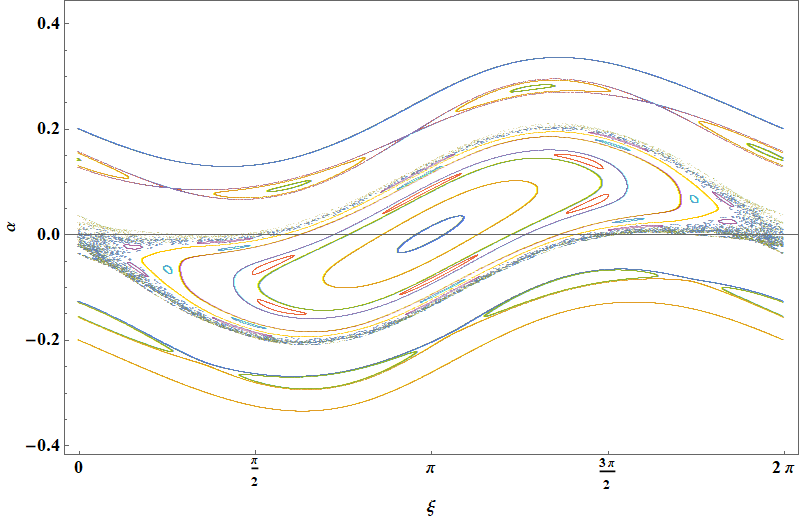} \qquad 
	\includegraphics[width=0.4\textwidth]{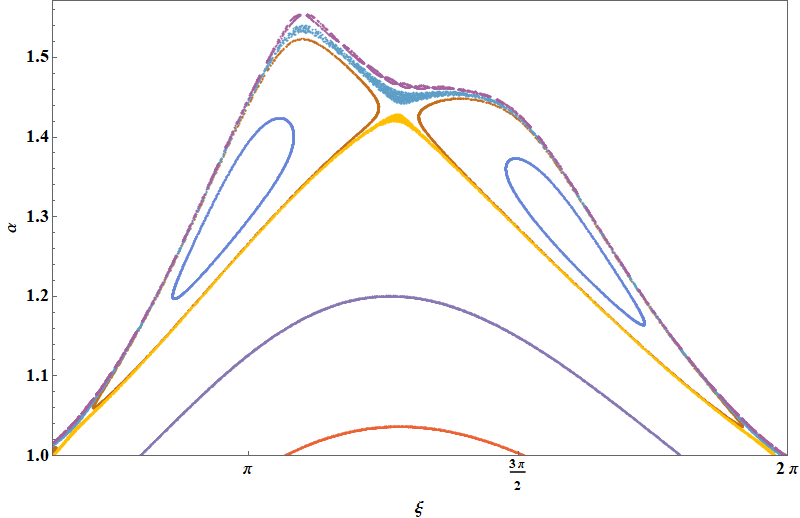}
	\caption{Phase portrait for the non-centred circular Keplerian refractive billiard with fixed physical parameters $x_0=0.1$, $h_E=9$, $\mu=2$ and $\omega=1$. In the first line we have $h_I=h_E+0$ and we note a more complex dynamics due to the presence of secondary islands. 
	In the second line $h_I=h_E+3$ and chaos arises. In the third line $h_I=h_E+6$  the dynamics is more complex; we propose a comparison with the right picture with Figure \ref{fig:xc002} top right.}
	\label{fig:xc01}
\end{figure}

\begin{figure}[h!]
	\includegraphics[width=0.45\textwidth]{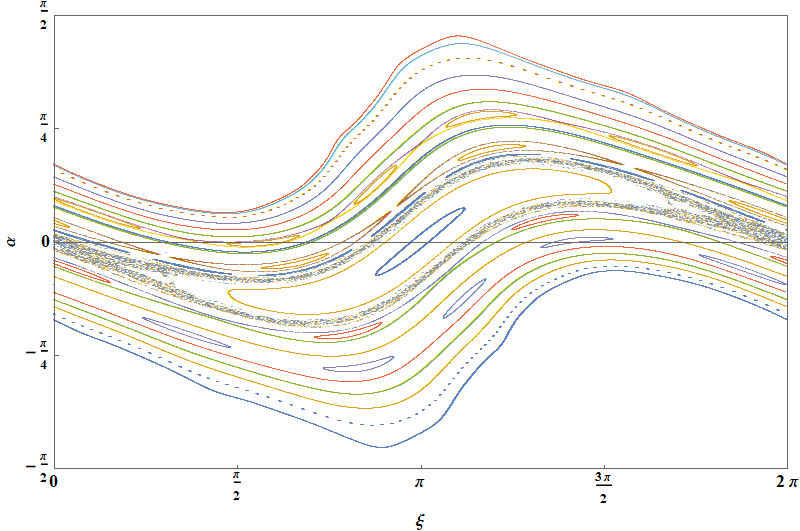}
	\caption{Phase portrait for the non-centred circular Keplerian refractive billiard with physical parameters $x_0=0.4$, $h_E=9$, $\mu=2$ $\omega=1$ $h_I=h_E+0$.}
	\label{fig:xc04}
\end{figure}

Let us now pass to the refractive model: here, except for the focused circular case, no other domains have been proved analytically to lead to an integrable dynamics. It does make sense, then, to start our analysis from the focused elliptic case (which, we remind, is integrable in the reflective case). Figure \ref{fig:refractiveFocused} shows how, in such case, possibly local chaotic phenomena are visible even for very low values of the inner energy and eccentricity. In particular, one can observe how, on the same ellipse, diffusive phenomena are already present locally around the saddle point in $(0,0)$ while, taking higher values of $h_I$, the dynamics is enriched by the presence of secondary islands. \\
We highlight the importance of such numerical result in the framework of refractive billiards: we are in fact showing numerically that, unlike the reflective system, in the refraction case the system is chaotic \emph{even in the case of a focused ellipse}, marking an important difference between the two models. Since diffusion is already evident in the non centred circular case, an even more complex chaotic behaviour is expected in the non-focused elliptic case. In fact, numerical evidences of the onset of chaos in this case is immediately obtained for the same parameters used in Figure \ref{fig:refractiveFocused}. 

\begin{figure}[h!]
	\includegraphics[width=0.4\textwidth]{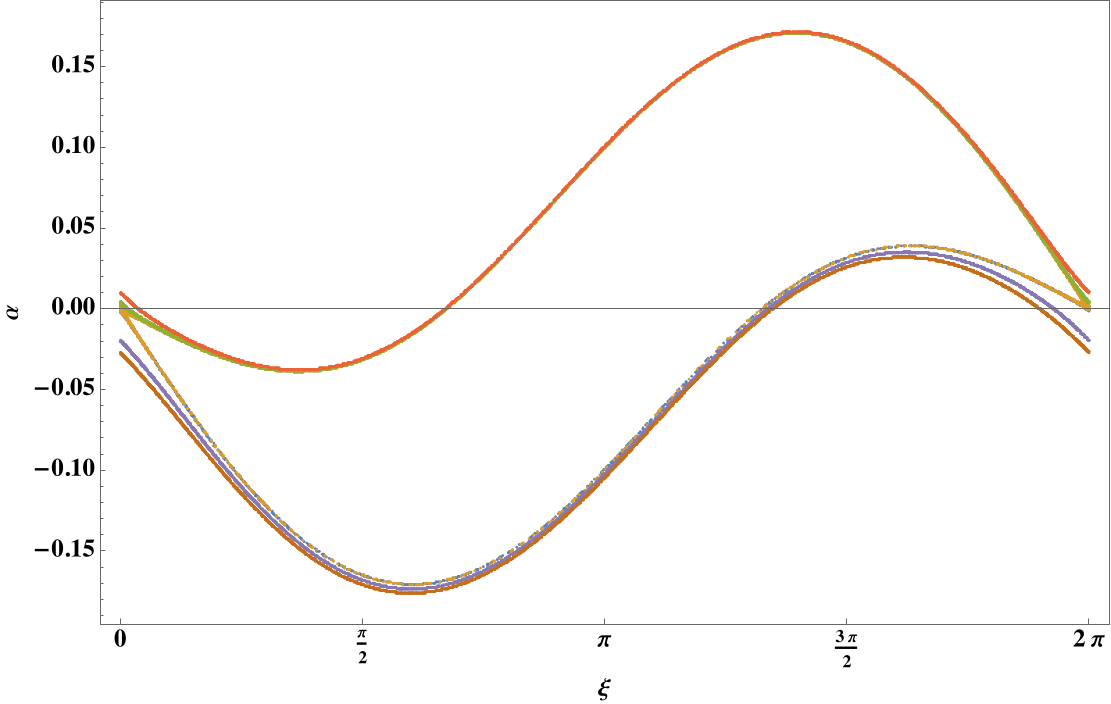}\qquad
	\includegraphics[width=0.4\textwidth]{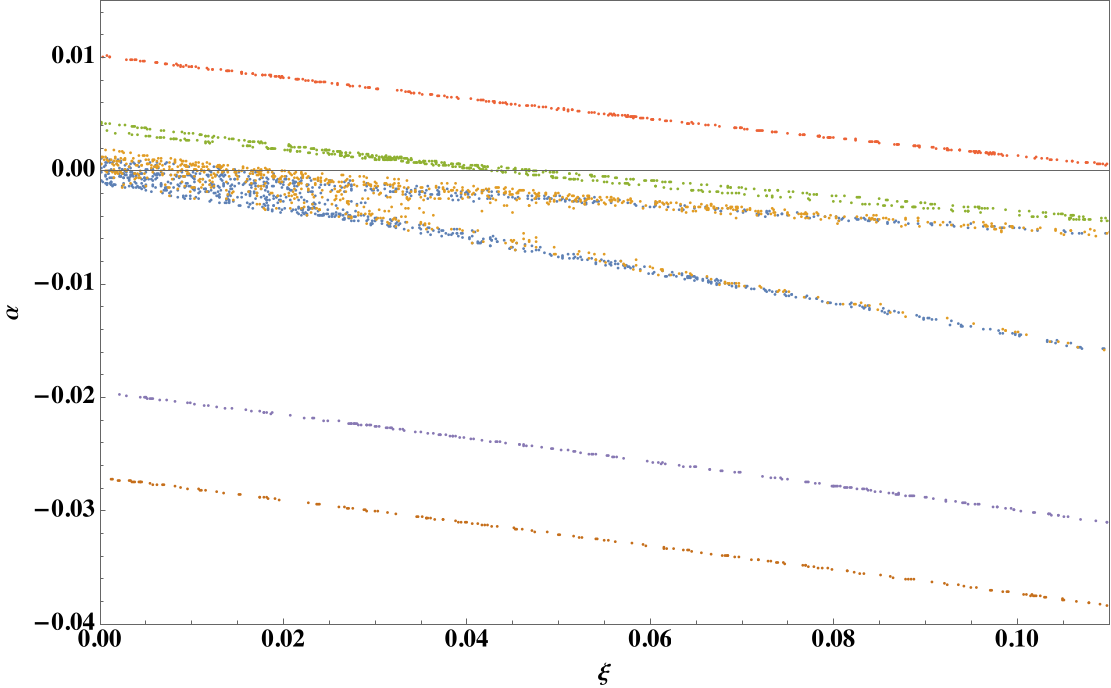}\\
	\includegraphics[width=0.4\textwidth]{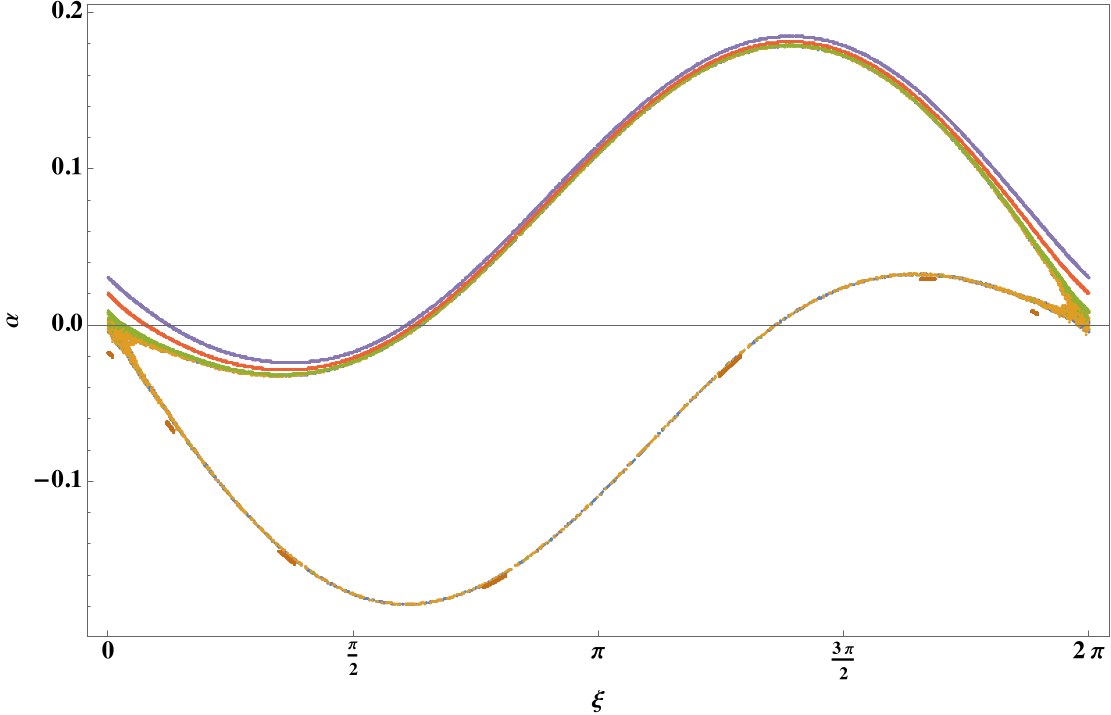}\qquad
	\includegraphics[width=0.4\textwidth]{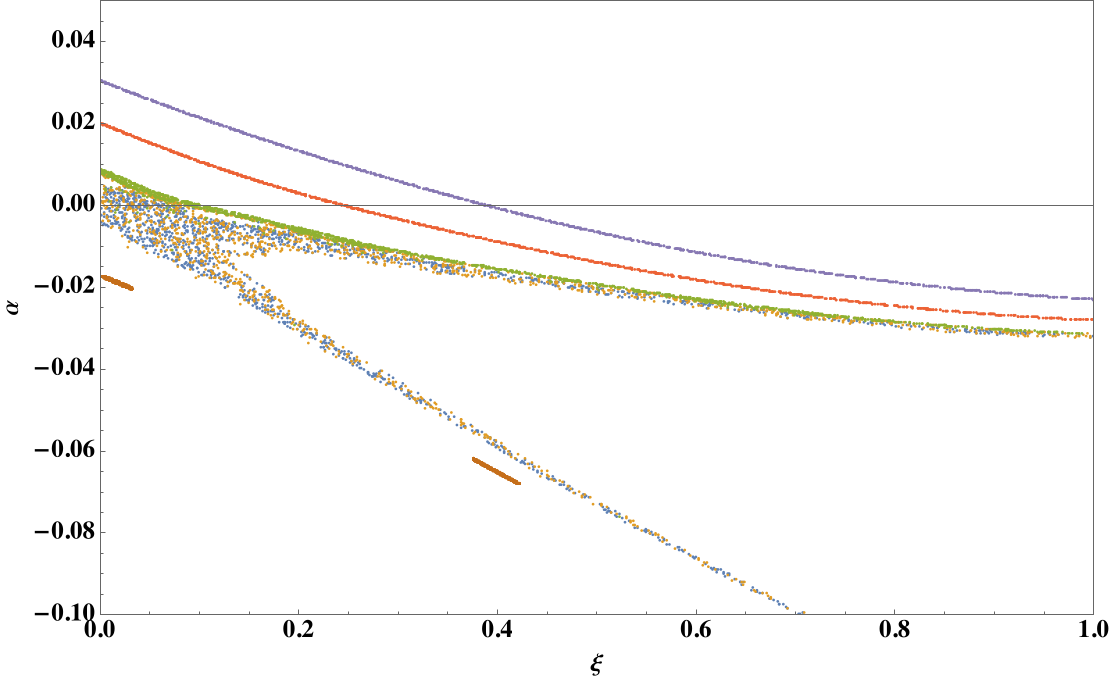}
	\caption{Phase portrait of the first return map in the refractive case, when the domain is a focused ellipse with eccentricity $e=0.1$. The physical parameters are $h_E=9$, $\omega=1$, $\mu=2$, $h_I=h_E+1$ (first row) and $h_I=h_E+2$ (second row). Left: orbits in the whole phase space; right: zoom around the saddle fixed point in $(0,0)$.  }
	\label{fig:refractiveFocused}
\end{figure}

\section{Conclusions}
As already pointed out in Section \ref{sec:intro}, the present paper has the aim to enrich the analysis carried on \cite{deblasiterraciniellissi, IreneSusNew,IreneSusViNEW}, considering in details the reflective Keplerian case as well. The analytical results in Theorems \ref{thm:primo_teorema} and \ref{thm:DeltaKep}, give simple condition to study either  the admissibility of an elliptic boundary or to detect the stability of an homothetic trajectory in the reflective case. They are linked by the fact that, in different ways, both theorems highlight the presence of \emph{bifurcation} phenomena, achieved with very simple changes in the position of the mass inside $D$ or the values of physical parameters. \\
The simulations presented in Section \ref{sec:numerical} have a double scope: first of all, they are useful to explore cases that we are not able, at present, to study analytically. Secondly, they provide actual numerical threshold for the physical and geometrical parameters which allow to observe chaotic behaviour. What we can conclude by observing the above results is that, except for the cases already proved to be integrable (focused ellipse for the reflective system and centred circle for both of them), chaotic phenomena seem to always appear. This fact, non-trivial whenever we are working at low inner energies or inside a non admissible domain, is particularly interesting and motivates our purpose of further analysing such cases, searching for more general analytical conditions that guarantee chaos. 
\newpage

\end{document}